\documentclass[11pt]{article}
\usepackage{amssymb,amsmath}
\usepackage[mathscr]{eucal}
\usepackage[cm]{fullpage}
\usepackage[english]{babel}
\usepackage[latin1]{inputenc}
\usepackage{tikz}
\usepackage{amsthm}
\usepackage{soul} 
\usepackage{cancel} 

\def\dom{\mathop{\mathrm{Dom}}\nolimits}
\def\im{\mathop{\mathrm{Im}}\nolimits}

\def\id{\mathrm{id}}

\def\PT{\mathcal{PT}}
\def\T{\mathcal{T}}
\def\I{\mathcal{I}}
\def\Sym{\mathcal{S}}

\newcommand{\PEnd}{\mathrm{PEnd}}
\newcommand{\PwEnd}{\mathrm{PwEnd}}
\newcommand{\PsEnd}{\mathrm{PsEnd}}
\newcommand{\PswEnd}{\mathrm{PswEnd}}
\newcommand{\PAut}{\mathrm{PAut}}
\newcommand{\IEnd}{\mathrm{IEnd}}

\newcommand{\transf}[1]{\left(\begin{smallmatrix}#1\end{smallmatrix}\right)}

\newtheorem{theorem}{Theorem}[section]
\newtheorem{proposition}[theorem]{Proposition}

\newtheorem{lemma}[theorem]{Lemma}

\tikzset{
  vertex/.style={
    circle,
    minimum size=2mm,
    fill,
    inner sep=0,
    outer sep=0,
  },
  edge/.style={
    line width=.2mm,
  }
}

\newcommand{\lastpage}{\addresss}

\newcommand{\addresss}{\small {\sf

\noindent{\sc Ilinka Dimitrova},
Department of Mathematics,
Faculty of Mathematics and Natural Science,
South-West University "Neofit Rilski",
2700 Blagoevgrad,
Bulgaria;
e-mail: ilinka\_dimitrova@swu.bg.

\medskip

\noindent{\sc V\'\i tor H. Fernandes},
Center for Mathematics and Applications (NOVA Math)
and Department of Mathematics,
Faculdade de Ci\^encias e Tecnologia,
Universidade Nova de Lisboa,
Monte da Caparica,
2829-516 Caparica,
Portugal;
e-mail: vhf@fct.unl.pt.

\medskip

\noindent{\sc J\"{o}rg Koppitz},
Institute of Mathematics and Informatics,
Bulgarian Academy of Sciences,
1113 Sofia,
Bulgaria;
e-mail: koppitz@math.bas.bg.
}}

\title{Presentations for monoids of partial endomorphisms of a star graph}

\author{Ilinka Dimitrova\footnote{This work was supported by the project BG05M2OP001-2.016-0018 ``MODERN-A: Modernization in partnership through digitalization of the Academic ecosystem".},\, V\'\i tor H. Fernandes$^{*,}$\footnote{This work is funded by national funds through the FCT - Funda\c c\~ao para a Ci\^encia e a Tecnologia, I.P., under the scope of the (Center for Mathematics and Applications) projects UIDB/00297/2020 (https://doi.org/10.54499/UIDB/00297/2020) and UIDP/00297/2020 (https://doi.org/10.54499/UIDP/00297/2020).}~ and  J\"org Koppitz
}

\begin{document}

\maketitle

\begin{abstract}
In this paper, we consider the monoids of all partial endomorphisms, of all partial weak endomorphisms,
of all injective partial endomorphisms,
of all partial strong endomorphisms and of all partial strong weak endomorphisms
of a star graph with a finite number of vertices.
Our main objective is to exhibit a presentation for each of them.
\end{abstract}

\medskip

\noindent{\small 2020 \em Mathematics subject classification: \em 20M20, 20M05, 05C12, 05C25.}

\noindent{\small\em Keywords: \em transformations, partial endomorphisms, star graphs, presentations.}

\section{Introduction}\label{presection}

Let $X$ be a set and denote by $X^*$ the free monoid generated by
$X$. Often, in this context, the set $X$ is called an \textit{alphabet} and its elements are called \textit{words}.
A \textit{monoid presentation} is an ordered pair $\langle X\mid
R\rangle$, where $X$ is an alphabet and $R$ is a subset of
$X^*\times X^*$. An element $(u,v)$ of $X^*\times X^*$ is called a
\textit{relation} of $X^*$ and it is usually represented by $u=v$.
A monoid $M$ is said to be \textit{defined by a presentation} $\langle X\mid R\rangle$ if $M$ is
isomorphic to $X^*/{\sim_R}$, where $\sim_R$ denotes the congruence on $A^*$ generated by $R$,
i.e. $\sim_R$ is the smallest
congruence on $X^*$ containing $R$.
Suppose that $X$ is a generating set of a monoid $M$ and let $u=v$ be a relation of $X^*$.
We say that a relation $u=v$ of $X^*$ is \textit{satisfied} by $X$ if $u=v$ is an equality in $M$.
For more details, see \cite{Lallement:1979} or \cite{Ruskuc:1995}.

A well-known direct method to obtain a presentation for a monoid, that we might consider folklore,
is given by the following result; see, for example, \cite[Proposition 1.2.3]{Ruskuc:1995}.

\begin{proposition}\label{provingpresentation}
Let $M$ be a monoid generated by a set $X$.
Then, $\langle X\mid R\rangle$ is a presentation for $M$ if and only
if the following two conditions are satisfied:
\begin{enumerate}
\item
The generating set $X$ of $M$ satisfies all the relations from $R$;
\item
If $w_1,w_2\in A^*$ are any two words such that
the generating set $X$ of $M$ satisfies the relation $w_1=w_2$, then $w_1\sim_R w_2$.
\end{enumerate}
\end{proposition}

For finite monoids, a standard method to find a presentation
is described by the following result, adapted for the monoid case from
\cite[Proposition 3.2.2]{Ruskuc:1995}.

\begin{theorem}[Guess and Prove method] \label{ruskuc}
Let $M$ be a finite monoid, let $X$ be a generating set for $M$,
let $R\subseteq X^*\times X^*$ be a set of relations and let
$W\subseteq X^*$. Assume that the following conditions are
satisfied:
\begin{enumerate}
\item The generating set $X$ of $M$ satisfies all the relations from $R$;
\item For each word $w\in X^*$, there exists a word $w'\in W$ such that $w\sim_R w'$;
\item $|W|\le|M|$.
\end{enumerate}
Then, $M$ is defined by the presentation $\langle X\mid R\rangle$.
\end{theorem}

Notice that, if $W$ satisfies the above conditions, then, in fact,
$|W|=|M|$.

\smallskip

Let $X$ be an alphabet, let $R\subseteq X^*\times X^*$ be a set of
relations and let $W$ be a subset of $X^*$. We say that $W$ is a set of
\textit{canonical forms} for a finite monoid $M$
if Conditions 2 and 3 of Theorem \ref{ruskuc} are satisfied. Suppose
that the empty word belongs to $W$ and, for each letter $x\in X$ and
for each word $w\in W$, there exists a word $w'\in W$ such that $wx\sim_R w'$.
Then, it is easy to show that $W$ satisfies Condition 2.

In this paper, we will mainly use the method given by Theorem \ref{ruskuc} and the condition described above instead of Condition 2.

\smallskip

For a set $R$ of relations on an alphabet $A$, we also denote by $\sim_R$ the congruence on $B^*$ generated by $R$, for any alphabet $B$ containing $A$.

\medskip

Next, for a set $\Omega$, denote by $\PT(\Omega)$ the monoid (under composition) of all
partial transformations on $\Omega$, by
$\T(\Omega)$ the submonoid of $\PT(\Omega)$ of all full transformations on $\Omega$,
by $\I(\Omega)$ the \textit{symmetric inverse monoid} on $\Omega$, i.e.
the inverse submonoid of $\PT(\Omega)$ of all
partial permutations on $\Omega$,
and by $\Sym(\Omega)$ the \textit{symmetric group} on $\Omega$,
i.e. the subgroup of $\PT(\Omega)$ of all
permutations on $\Omega$.

\smallskip

Let $\Omega$ be a finite set.
Towards the end of the 19th century (1987),
a presentation for the symmetric group $\Sym(\Omega)$ was determined by Moore \cite{Moore:1897}.
Over half a century later (1958), A\u{\i}zen\v{s}tat \cite{Aizenstat:1958}  gave a presentation for
full transformation monoid $\T(\Omega)$.
A few years later (1961),
presentations for the partial transformation monoid $\PT(\Omega)$
and for the symmetric inverse monoid $\I(\Omega)$
were found by Popova \cite{Popova:1961}.
In 1962, A\u{\i}zen\v{s}tat \cite{Aizenstat:1962} and Popova \cite{Popova:1962} exhibited presentations for the monoids of
all order-preserving transformations and of all order-preserving partial transformations of a finite chain, respectively, and from the sixties until our days several authors obtained presentations for many classes of monoids.
Further examples can be found, for instance, in
\cite{
Cicalo&al:2015,
East:2011,
Feng&al:2019,
Fernandes:2001,
Fernandes:2002survey,
Fernandes&Gomes&Jesus:2004,
Fernandes&Quinteiro:2016,
Howie&Ruskuc:1995,
Ruskuc:1995}.

\medskip

Now, let $G=(V,E)$ be a simple graph (i.e. undirected, without loops or multiple edges).
We say that $\alpha\in\PT(V)$ is:
\begin{itemize}
\item a \textit{partial endomorphism} of $G$ if $\{u,v\}\in E$ implies  $\{u\alpha,v\alpha\}\in E$, for all $u,v\in\dom(\alpha)$;
\item a \textit{partial weak endomorphism} of $G$ if $\{u,v\}\in E$ and $u\alpha\ne v\alpha$ imply  $\{u\alpha,v\alpha\}\in E$, for all $u,v\in\dom(\alpha)$;
\item a \textit{partial strong endomorphism} of $G$ if $\{u,v\}\in E$ if and only if  $\{u\alpha,v\alpha\}\in E$, for all $u,v\in\dom(\alpha)$;
\item a \textit{partial strong weak endomorphism} of $G$ if $\{u,v\}\in E$ and $u\alpha\ne v\alpha$ if and only if $\{u\alpha,v\alpha\}\in E$, for all $u,v\in\dom(\alpha)$;
\item a \textit{partial automorphism} of $G$ if $\alpha$ is an injective mapping (i.e. a partial permutation) and $\alpha$ and $\alpha^{-1}$ are both partial endomorphisms.
\end{itemize}
It is easy to show that $\alpha\in\I(V)$ is a partial automorphism of $G$ if and only if $\alpha$ is a partial strong endomorphism of $G$.
Let us denote by:
\begin{itemize}
\item $\PEnd(G)$ the set of all partial endomorphisms of $G$;
\item $\IEnd(G)$ the set of all injective partial endomorphisms of $G$;
\item $\PwEnd(G)$ the set of all partial weak endomorphisms of $G$;
\item $\PsEnd(G)$ the set of all partial strong endomorphisms of $G$;
\item $\PswEnd(G)$ the set of all partial strong weak endomorphisms of $G$;
\item $\PAut(G)$ the set of all partial automorphisms of $G$.
\end{itemize}

Clearly, $\PEnd(G)$, $\IEnd(G)$, $\PwEnd(G)$, $\PsEnd(G)$, $\PswEnd(G)$ and $\PAut(G)$ are submonoids of $\PT(V)$.
Moreover, we have the following Hasse diagram for the set inclusion relation:
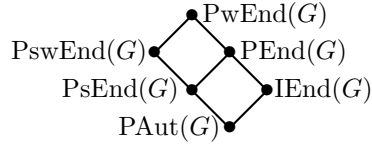
\begin{center}
\begin{tikzpicture}[scale=0.5]
\draw (1,0) node{$\bullet$} (0,1) node{$\bullet$} (-1,2) node{$\bullet$} (1,2) node{$\bullet$} (0,3) node{$\bullet$} (2,1) node{$\bullet$};
\draw (-0.6,0) node{\small$\PAut(G)$} (-1.9,1.0) node{\small$\PsEnd(G)$}
(-3.0,2.0) node{\small$\PswEnd(G)$} (2.7,2.0) node{\small$\PEnd(G)$}
(1.95,3.0) node{\small$\PwEnd(G)$} (3.5,1.0) node{\small$\IEnd(G)$};
\draw[thick] (1,0) -- (0,1);
\draw[thick] (0,1) -- (-1,2);
\draw[thick] (0,1) -- (1,2);
\draw[thick] (-1,2) -- (0,3);
\draw[thick] (1,2) -- (0,3);
\draw[thick] (1,0) -- (2,1);
\draw[thick] (2,1) -- (1,2);
\end{tikzpicture}
\end{center}
(these inclusions may not be strict).
Notice that, $\PAut(G)$ is also an inverse submonoid of $\I(V)$.

\smallskip

Monoids of endomorphisms of graphs have many important applications, particularly related to automata theory; see \cite{Kelarev:2003}.
A large number of interesting results concerning graphs and algebraic properties of their endomorphism monoids have been obtained by several authors (see, for example, \cite{Bottcher&Knauer:1992,Fan:1996,Gu&Hou:2016,Hou&Luo&Fan:2012,Kelarev&Praeger:2003,Knauer:2011,Knauer&Wanichsombat:2014, Li:2003,Wilkeit:1996}).
In recent years, the authors together with T. Quinteiro, also have studied such kinds of monoids, in particular by considering finite undirected paths and cycles;
see \cite{Dimitrova&Fernandes&Koppitz&Quinteiro:2020,Dimitrova&Fernandes&Koppitz&Quinteiro:2021,Dimitrova&Fernandes&Koppitz&Quinteiro:2023arxiv}.

\smallskip

For any non negative integer $n$,  let $\Omega_n=\{1,2,\ldots,n\}$ and $\Omega_n^0=\{0, 1,\ldots,n\}=\{0\}\cup\Omega_n$.
Notice that, $\Omega_0=\emptyset$ and $\Omega_0^0=\{0\}$.
For an integer $n\geqslant1$, consider the \textit{star graph}
$$
S_n=(\Omega_{n-1}^0, \{\{0,i\}\mid 1\leqslant i\leqslant n-1\})
$$
with $n$ vertices.
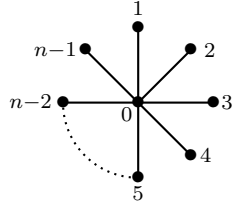
\begin{center}
\begin{tikzpicture}
\draw (0,0) node{$\bullet$} (0,2) node{$\bullet$} (-1,1) node{$\bullet$} (1,1) node{$\bullet$};
\draw (0.7,0.3) node{$\bullet$} (0.7,1.7) node{$\bullet$} (-0.7,1.7) node{$\bullet$};
\draw (0,-0.2) node{$\scriptstyle5$} (0,2.25) node{$\scriptstyle1$} (-1.43,1) node{$\scriptstyle n-2$} (1.18,1) node{$\scriptstyle3$};
\draw (0.9,0.3) node{$\scriptstyle4$} (0.95,1.7) node{$\scriptstyle2$} (-1.1,1.7) node{$\scriptstyle n-1$};
\draw (0,1) node{$\bullet$}; \draw (-.15,.85) node{$\scriptstyle0$};
\draw[thick] (0,1) -- (0,0); \draw[thick] (0,1) -- (0,2); \draw[thick] (0,1) -- (-1,1); \draw[thick] (0,1) -- (1,1);
\draw[thick] (0,1) -- (0.7,0.3); \draw[thick] (0,1) -- (0.7,1.7) ; \draw[thick] (0,1) -- (-0.7,1.7);
\draw[thick,dotted] (0,0) arc (-90:-180:1);
\end{tikzpicture}
\end{center}

These very basic graphs, which are special examples of \textit{trees} and also of \textit{complete bipartite graphs}, play a significant role in Graph Theory,
for instance, through the notions of \textit{star chromatic number} and \textit{star arboricity}.
We may also find important applications of star graphs in Computer Science.
For example, in Distributed Computing the \textit{star network} is one of the most common computer network topologies.

In 2023, Fernandes and Paulista \cite{Fernandes&Paulista:2023} considered the monoid of all partial isometries of the star graph $S_n$, which coincides with $\PAut(S_n)$, as observed by the authors in \cite{Dimitrova&Fernandes&Koppitz:2024}.
They determined the rank (i.e. the minimum size of a generating set) and the cardinality of this monoid as well as described its Green's relations and exhibited a presentation.

On the other hand, recently, the authors considered in \cite{Dimitrova&Fernandes&Koppitz:2024}
the monoids $\PsEnd(S_n)$, $\PswEnd(S_n)$, $\PEnd(S_n)$,
$\PwEnd(S_n)$ and $\IEnd(S_n)$. For each of these monoids, they studied their regularity, described their Green's relations and computed their cardinalities and ranks. In this paper, our main aim is to determine presentations for each of them.

\smallskip

For general background on Semigroup Theory and standard notations, we would like to refer the reader to Howie's book \cite{Howie:1995}.
Regarding Algebraic Graph Theory, we refer to Knauer's book \cite{Knauer:2011}.
We would also like to point out that we made massive use of computational tools, namely GAP \cite{GAP4}.

\section{On partial endomorphisms of a star graph}\label{prelim}

In this section, we recall known facts about the monoids $\PsEnd(S_n)$, $\PswEnd(S_n)$, $\PEnd(S_n)$, $\PwEnd(S_n)$, $\PAut(S_n)$ and $\IEnd(S_n)$ that we will use later in this paper. For more details, see \cite{Dimitrova&Fernandes&Koppitz:2024} and \cite{Fernandes&Paulista:2023}.

\smallskip

Let $n\geqslant3$.
Let us consider the following transformations of $\PT(\Omega_{n-1})$:
$$
a=\begin{pmatrix}
1 & 2 & 3 & \cdots & n-1 \\
2 & 1 & 3 &\cdots & n-1
\end{pmatrix}, \quad
b=\begin{pmatrix}
1 & 2 & \cdots & n-2 &n-1 \\
2 & 3 & \cdots & n-1 & 1
\end{pmatrix}, \quad
e=\begin{pmatrix}
1 & 2 & 3 & \cdots & n -1\\
1 & 1 & 3 & \cdots & n-1
\end{pmatrix}
$$
and
$$
f=\begin{pmatrix}
2 & 3 & \cdots & n-1 \\
2 & 3 & \cdots & n-1
\end{pmatrix}.
$$
Then, it is well known that $\{a,b,e,f\}$ is a set of generators of $\PT(\Omega_{n-1})$ (with minimum size for $n\geqslant4$).

\smallskip

Consider the mapping $\zeta:\PT(\Omega_{n-1}^0)\longrightarrow\PT(\Omega_{n-1}^0)$, $\alpha\longmapsto\zeta_\alpha$,
defined by $\dom(\zeta_\alpha)=\dom(\alpha)\cup\{0\}$, $0\zeta_\alpha=0$ and
$\zeta_\alpha|_{\Omega_{n-1}}=\alpha|_{\Omega_{n-1}}$,
for any $\alpha\in\PT(\Omega_{n-1}^0)$.
This mapping is neither a homomorphism nor injective.
Despite this, $\zeta|_{\PT(\Omega_{n-1})}$ and $\zeta|_{\I(\Omega_{n-1})}$ are injective homomorphisms and so
$$
\PT(\Omega_{n-1})\simeq\PT(\Omega_{n-1})\zeta=
\{\alpha\in\PT(\Omega_{n-1}^0)\mid\mbox{$0\in\dom(\alpha)$, $0\alpha=0$ and $\Omega_{n-1}\alpha\subseteq\Omega_{n-1}$}\}
$$
and
$$
\I(\Omega_{n-1})\simeq\I(\Omega_{n-1})\zeta=
\{\alpha\in\I(\Omega_{n-1}^0)\mid\mbox{$0\in\dom(\alpha)$, $0\alpha=0$}\}.
$$

\smallskip

Define
$$
a_0=\zeta_a=\begin{pmatrix}
0&1 & 2 & 3 & \cdots & n-1 \\
0&2 & 1 & 3 &\cdots & n-1
\end{pmatrix}, \quad
b_0=\zeta_b=\begin{pmatrix}
0&1 & 2 & \cdots & n-2 &n-1 \\
0&2 & 3 & \cdots & n-1 & 1
\end{pmatrix},
$$
$$
e_0=\zeta_e=\begin{pmatrix}
0&1 & 2 & 3 & \cdots & n -1\\
0&1 & 1 & 3 & \cdots & n-1
\end{pmatrix}
\quad\text{and}\quad
f_0=\zeta_f=\begin{pmatrix}
0&2 & 3 & \cdots & n-1 \\
0&2 & 3 & \cdots & n-1
\end{pmatrix}.
$$
Then, $\{a_0,b_0,e_0,f_0\}$ is a set of generators of $\PT(\Omega_{n-1})\zeta$ (with minimum size for $n\geqslant4$).
Moreover, we have $a_0,b_0,e_0,f_0\in\PwEnd(S_n)$.
Let us also consider the following transformations of $\PwEnd(S_n)$:
$$
c=\begin{pmatrix}
1 & 2 & \cdots & n-1 \\
0 & 2 &\cdots & n-1
\end{pmatrix}, \quad
c_0=\zeta_c=\begin{pmatrix}
0 & 1 & 2 & \cdots & n-1 \\
0 & 0 & 2 &\cdots & n-1
\end{pmatrix},
$$
$$
d=\begin{pmatrix}
1 & 2 & \cdots &n-1 \\
1 & 2 & \cdots & n-1
\end{pmatrix}, \quad
z=\begin{pmatrix}
0 & 1 & 2 & \cdots & n -1\\
1 & 0 & 0 & \cdots & 0
\end{pmatrix}
\quad\text{and}\quad
z_0=\zeta_z=\begin{pmatrix}
0 & 1 & 2 & \cdots & n -1\\
0 & 0 & 0 & \cdots & 0
\end{pmatrix}.
$$

\smallskip

Let $2\PT_{n-1}=\PT(\Omega_{n-1})\zeta\cup\PT(\Omega_{n-1})$.
Observe that $2\PT_{n-1}$ is the submonoid of $\PT(\Omega_{n-1}^0)$ generated by $\{a_0,b_0,e_0,f_0,d\}$ and
$
|2\PT_{n-1}|=2|\PT(\Omega_{n-1})|=2n^{n-1}.
$

\smallskip

Therefore, we have:
\begin{align*}
\PsEnd(S_n)&=2\PT_{n-1} \cup \{\alpha\in\PT(\Omega_{n-1}^0)\mid \mbox{$0\not\in\dom(\alpha)$ and $\im(\alpha)=\{0\}$}\} \\
&\cup \{\alpha\in\PT(\Omega_{n-1}^0)\mid \mbox{$0\in\dom(\alpha)$, $0\alpha\neq0$ and $\Omega_{n-1}\alpha\subseteq\{0\}$}\},
\end{align*}
$|\PsEnd(S_n)|= 2n^{n-1} + n2^{n-1} -1$
and $\{a_0,b_0,e_0,f_0,d,z\}$ is a generating set of the monoid $\PsEnd(S_n)$;
$$
\PswEnd(S_n)=\PsEnd(S_n)\cup
\{\alpha\in\PT(\Omega_{n-1}^0)\mid \mbox{$0\in\dom(\alpha)$ and $|\im(\alpha)|=1$}\},
$$
$|\PswEnd(S_n)|=2n^{n-1} + n2^n -n-1$
and the set $\{a_0,b_0,e_0,f_0,d,z,z_0\}$ generates the monoid $\PswEnd(S_n)$;
$$
\PEnd(S_n)=\PsEnd(S_n)\cup
\{\alpha\in\PT(\Omega_{n-1}^0)\mid \mbox{$0\not\in\dom(\alpha)$, $0\in\im(\alpha)$ and $\im(\alpha)\cap\Omega_{n-1}\neq\emptyset$}\},
$$
$|\PEnd(S_n)|=(n+1)^{n-1}+n^{n-1}+(n-1)2^{n-1}$
and the set $\{a_0,b_0,e_0,f_0,c,d,z\}$ generates the monoid $\PEnd(S_n)$;
\begin{align*}
\PwEnd(S_n)&=\PEnd(S_n)\cup
\{\alpha\in\PT(\Omega_{n-1}^0)\mid \mbox{$0\in\dom(\alpha)$, $0\alpha=0$ and $0\in\Omega_{n-1}\alpha$}\}\\
&\cup
\{\alpha\in\PT(\Omega_{n-1}^0)\mid \mbox{$\{0\}\subsetneq\dom(\alpha)$, $0\alpha\neq0$ and $0\alpha\in\Omega_{n-1}\alpha\subseteq\{0,0\alpha\}$}\}
\end{align*}
$|\PwEnd(S_n)|=2(n+1)^{n-1}+(n-1)3^{n-1}$
and $\{a_0,b_0,e_0,f_0,c_0,d,z\}$ is a generating set of the monoid $\PwEnd(S_n)$.

\smallskip

Finally, regarding injective partial endomorphisms, let
$$
e_1=\begin{pmatrix}0 & 1 & \cdots & n-2 \\ 0 & 1 & \cdots & n-2\end{pmatrix}
\quad\text{and}\quad
z_1=\begin{pmatrix}0 & 1 \\  1 & 0\end{pmatrix}.
$$
It is well known that  $\left\{a, b, \left(\begin{smallmatrix}1 & 2 & \cdots & n-2 \\ 1 & 2 & \cdots & n-2\end{smallmatrix}\right)\right\}$
generates $\I(\Omega_{n-1})$ (for example, see \cite{Howie:1995}).
Then, $\{a_0,b_0,e_1\}= \left\{\zeta_a, \zeta_b, {\left(\begin{smallmatrix}1 & 2 & \cdots & n-2 \\ 1 & 2 & \cdots & n-2\end{smallmatrix}\right)}\zeta\right\}$,
 $\{a_0,b_0,e_1,d,z_1\}$ and $\{a_0,b_0,e_1,c,d,z_1\}$ are generating sets of the monoids $\I(\Omega_{n-1})\zeta$, $\PAut(S_n)$ and $\IEnd(S_n)$, respectively.
Recall that
$$
\IEnd(S_n)=\PAut(S_n)\cup K_0,
$$
where $K_0=\{\alpha\in\I(\Omega_{n-1}^0)\mid \mbox{$0\not\in\dom(\alpha)$, $0\in\im(\alpha)$ and $\im(\alpha)\cap\Omega_{n-1}\neq\emptyset$}\}$,
$|K_0|=\sum_{k=2}^{n-1}\binom{n-1}{k}\binom{n-1}{k-1}k!$,
$|\PAut (S_n)|=1+n^2+2\sum_{k=1}^{n-1}\binom{n-1}{k}^2k!$ and
$|\IEnd(S_n)|=3+3n^2-4n+\sum_{k=2}^{n-1}\left(\binom{n}{k}+\binom{n-1}{k}\right)\binom{n-1}{k}k!$.

\smallskip

Notice also that, for $n\geqslant4$, all these generating sets have minimum size.

\section{A presentation for $\PsEnd(S_n)$}

Let $R_0$ be a set of relations defining $\PT(\Omega_{n-1})\zeta$ on the generators $a_0$, $b_0$, $e_0$ and $f_0$ and let $R_d$ be the set $R_0$ together with the relations
$$
d^2=d,~ a_0d=da_0,~ b_0d=db_0,~ e_0d=de_0 ~\text{and}~ f_0d=df_0
$$
on the alphabet $\{a_0,b_0,e_0,f_0,d\}$.
Then, we have:

\begin{theorem}\label{pres2pt}
The presentation $\langle a_0,b_0,e_0,f_0,d\mid R_d\rangle$ defines the monoid $2\PT_{n-1}$.
\end{theorem}
\begin{proof}
First, observe that it is easy to check that all relations of $R_d$ are satisfied by the generators $a_0$, $b_0$, $e_0$, $f_0$ and $d$ of $2\PT_{n-1}$.

Next, let $w$ and $w'$ be two words of $\{a_0,b_0,e_0,f_0,d\}^*$ such that $w=w'$ in $2\PT_{n-1}$.
Through the relations $d^2=d$ and $xd=dx$ with $x\in \{a_0,b_0,e_0,f_0\}$, it is clear that $w\sim_{R_d}d^{i_1}w_1$ and $w'\sim_{R_d}d^{i_2}w_2$
for some $i_1,i_2\in\{0,1\}$ and $w_1,w_2\in \{a_0,b_0,e_0,f_0\}^*$.

Since we have $d^{i_1}w_1=w=w'=d^{i_2}w_2$ in $2\PT_{n-1}$, it easy to deduce that $i_1=i_2$.
If $i_1=i_2=0$, then we get immediately $w_1=w_2$ in $\PT(\Omega_{n-1})\zeta$.
On the other hand, if $i_1=i_2=1$, then $dw_1=dw_2$ in $2\PT_{n-1}$ implies that we get again $w_1=w_2$ in $\PT(\Omega_{n-1})\zeta$.
It follows that  $w_1\sim_{R_0}w_2$ and so $dw_1\sim_{R_0}dw_2$.
Thus, $w\sim_{R_d}w'$, which concludes the proof.
\end{proof}

Now, let $R_\mathrm{s}$ be the set of relations $R_d$ together with the relations
$$
a_0z=z,~ b_0z=z,~ e_0z=z,~ dz=zf_0,~ zd=(f_0b_0)^{n-1}z,~ z^2=(e_0b_0)^{n-3}e_0  ~\text{and}~  d(f_0b_0)^{n-1}z=d(f_0b_0)^{n-1}
$$
on the alphabet $\{a_0,b_0,e_0,f_0,d,z\}$.

Next, we present a very simple property that we will use repeatedly below without mentioning it explicitly.

\begin{lemma}
$z^3\sim_{R_\mathrm{s}} z$.
\end{lemma}
\begin{proof}
We have $z^3=z^2z\sim_{R_\mathrm{s}} (e_0b_0)^{n-3}e_0z \sim_{R_\mathrm{s}} z$, as required.
\end{proof}

Let $W_0\subseteq \{a_0,b_0,e_0,f_0\}^*$ be a set of canonical forms for $\PT(\Omega_{n-1})\zeta$ on the generators $a_0$, $b_0$, $e_0$ and $f_0$,
containing the empty word and the word $(f_0b_0)^{n-1}$.
For $1\leqslant i_1<\cdots<i_k\leqslant n-1$ with $1\leqslant k\leqslant n-1$, denote by $w_{i_1,\ldots, i_k}$ the word of $W_0$ representing the element
$\left(\begin{smallmatrix} 0 & i_1 & \cdots & i_k \\ 0 & 1 & \cdots & 1\end{smallmatrix}\right)$ of $\PT(\Omega_{n-1})\zeta$.
Let $W_1=dW_0$,
$$
W_2 = \{w_{i_1,\ldots, i_k}zb_0^{j-1} \mid  1\leqslant i_1<\cdots<i_k\leqslant n-1, ~ 1\leqslant j,k\leqslant n-1\},
$$
$$
W_3 = \{dw_{i_1,\ldots, i_k}z \mid  1\leqslant i_1<\cdots<i_k\leqslant n-1, ~ 1\leqslant k\leqslant n-1\},
$$
$$
W_4 = \{(f_0b_0)^{n-1}zb_0^{j-1} \mid  1\leqslant j\leqslant n-1\}
$$
and $W_\text{s}=W_0\cup W_1\cup W_2\cup W_3\cup W_4$.
Observe that
$$
|W_\mathrm{s}|=n^{n-1} + n^{n-1}  + (n-1)(2^{n-1}-1) + (2^{n-1}-1) + (n-1) = 2n^{n-1} + n2^{n-1} -1 = |\PsEnd(S_n)|.
$$
\begin{lemma}\label{sfun}
Let $w\in W_\mathrm{s}$ and $x\in \{a_0,b_0,e_0,f_0,d,z\}$. Then, there exists $w'\in W_\mathrm{s}$ such that $wx\sim_{R_\mathrm{s}} w'$.
\end{lemma}
\begin{proof}
Let us denote $\sim_{R_\mathrm{s}}$ simply by $\sim$.

\smallskip

We start the proof by considering $x=z$.

If $w\in W_0$, then $wz \sim wz^2z \sim w(e_0b_0)^{n-3}e_0z \sim_{R_0} w_{i_1,\ldots, i_k}z\in W_2$,
where $\{0<i_1<\cdots<i_k\}=\dom(w)$.

Now, let us take $w\in W_1$. Then, $w=dw_0$, for some $w_0\in W_0$.
From the previous case, we have $w_0z \sim w_{i_1,\ldots, i_k}z$, where $\{0<i_1<\cdots<i_k\}=\dom(w_0)$.
Then, $wz=dw_0z \sim d w_{i_1,\ldots, i_k}z\in W_3$.

Next, suppose that $w\in W_2$. Then, $w=w_{i_1,\ldots, i_k}zb_0^{j-1}$, for some $1\leqslant i_1<\cdots<i_k\leqslant n-1$ and $1\leqslant j,k\leqslant n-1$.
Hence, $wz=w_{i_1,\ldots, i_k}zb_0^{j-1}z \sim w_{i_1,\ldots, i_k}z^2 \sim w_{i_1,\ldots, i_k}(e_0b_0)^{n-3}e_0 \sim_{R_0} w_{i_1,\ldots, i_k}\in W_0$.

If $w\in W_3$, then $w=dw_{i_1,\ldots, i_k}z$, for some $1\leqslant i_1<\cdots<i_k\leqslant n-1$ and $1\leqslant k\leqslant n-1$, whence
$wz=dw_{i_1,\ldots, i_k}z^2 \sim dw_{i_1,\ldots, i_k}(e_0b_0)^{n-3}e_0 \sim_{R_0} dw_{i_1,\ldots, i_k}\in W_1$.

Finally, if $w\in W_4$, then $w=(f_0b_0)^{n-1}zb_0^{j-1}$ for some $1\leqslant j\leqslant n-1$, and so
$$
wz=(f_0b_0)^{n-1}zb_0^{j-1}z \sim (f_0b_0)^{n-1}z^2 \sim (f_0b_0)^{n-1}(e_0b_0)^{n-3}e_0 \sim_{R_0} (f_0b_0)^{n-1}\in W_0.
$$

\smallskip

From now on, we suppose that $x\in \{a_0,b_0,e_0,f_0,d\}$.

\smallskip

First, let us take $w\in W_0\cup W_1$. Since $W_0\cup W_1$ is, clearly, a set of canonical forms for $2\PT_{n-1}$ on the generators $a_0$, $b_0$, $e_0$, $f_0$ and $d$, by Theorem \ref{pres2pt}, there exists $w'\in W_0\cup W_1$ such that $wx\sim_{R_d} w'$. Hence, $w'\in W_\mathrm{s}$ and $wx\sim w'$.

\smallskip

Secondly, we suppose that $w\in W_2$.
Let $1\leqslant i_1<\cdots<i_k\leqslant n-1$ and $1\leqslant j,k\leqslant n-1$ be such that $w=w_{i_1,\ldots, i_k}zb_0^{j-1}$.

Suppose that $x=a_0$. Since we have
$$
(e_0b_0)^{n-3}e_0 b_0^{j-1} a_0 =
\left\{\begin{array}{ll}
(e_0b_0)^{n-3}e_0 b_0 & \mbox{if $j=1$} \\
(e_0b_0)^{n-3}e_0 b_0^0 & \mbox{if $j=2$} \\
(e_0b_0)^{n-3}e_0 b_0^{j-1} & \mbox{if $j\geqslant3$}
\end{array}\right.
$$
in $\PT(\Omega_{n-1})\zeta$, then
$$
(e_0b_0)^{n-3}e_0 b_0^{j-1} a_0 \sim_{R_0}
\left\{\begin{array}{ll}
(e_0b_0)^{n-3}e_0 b_0 & \mbox{if $j=1$} \\
(e_0b_0)^{n-3}e_0 b_0^0 & \mbox{if $j=2$} \\
(e_0b_0)^{n-3}e_0 b_0^{j-1} & \mbox{if $j\geqslant3$}
\end{array}\right.
$$
and so
\begin{equation}\label{z2}
z^2 b_0^{j-1} a_0 \sim
\left\{\begin{array}{ll}
z^2 b_0 & \mbox{if $j=1$} \\
z^2 b_0^0 & \mbox{if $j=2$} \\
z^2 b_0^{j-1} & \mbox{if $j\geqslant3$}.
\end{array}\right.
\end{equation}
Hence,
$$
wa_0 =  w_{i_1,\ldots, i_k}zb_0^{j-1}a_0 \sim w_{i_1,\ldots, i_k}z z^2 b_0^{j-1}a_0 \sim
\left\{\begin{array}{ll}
w_{i_1,\ldots, i_k}z z^2 b_0 & \mbox{if $j=1$} \\
w_{i_1,\ldots, i_k}z z^2 b_0^0 & \mbox{if $j=2$} \\
w_{i_1,\ldots, i_k}z z^2 b_0^{j-1} & \mbox{if $j\geqslant3$}
\end{array}\right.
\sim
\left\{\begin{array}{ll}
w_{i_1,\ldots, i_k}z b_0 \in W_2 & \mbox{if $j=1$} \\
w_{i_1,\ldots, i_k}z b_0^0 \in W_2 & \mbox{if $j=2$} \\
w_{i_1,\ldots, i_k}z b_0^{j-1} \in W_2 & \mbox{if $j\geqslant3$}.
\end{array}\right.
$$

Next, suppose that $x=b_0$. So, we have $wb_0=w_{i_1,\ldots, i_k}zb_0^{j-1}b_0\sim_{R_0} w_{i_1,\ldots, i_k}zb_0^{j'-1}\in W_2$,
where $j'=j+1$, if $1\leqslant j\leqslant n-2$, and $j'=1$, if $j=n-1$.

Now, let $x=e_0$. In this case, we have
$$
(e_0b_0)^{n-3}e_0 b_0^{j-1} e_0 =
\left\{\begin{array}{ll}
(e_0b_0)^{n-3}e_0 b_0^0 & \mbox{if $j=1,2$} \\
(e_0b_0)^{n-3}e_0 b_0^{j-1} & \mbox{if $j\geqslant3$}
\end{array}\right.
$$
in $\PT(\Omega_{n-1})\zeta$, then
$$
(e_0b_0)^{n-3}e_0 b_0^{j-1} e_0 \sim_{R_0}
\left\{\begin{array}{ll}
(e_0b_0)^{n-3}e_0 b_0^0 & \mbox{if $j=1,2$} \\
(e_0b_0)^{n-3}e_0 b_0^{j-1} & \mbox{if $j\geqslant3$}
\end{array}\right.
$$
and so
\begin{equation}\label{z2e}
z^2 b_0^{j-1} e_0 \sim
\left\{\begin{array}{ll}
z^2 b_0^0 & \mbox{if $j=1,2$} \\
z^2 b_0^{j-1} & \mbox{if $j\geqslant3$}.
\end{array}\right.
\end{equation}
Hence,
$$
we_0 =  w_{i_1,\ldots, i_k}zb_0^{j-1}e_0 \sim w_{i_1,\ldots, i_k}z z^2 b_0^{j-1}e_0
\!\sim\!
\left\{\begin{array}{ll}
w_{i_1,\ldots, i_k}z z^2 b_0^0 & \mbox{if $j=1,2$} \\
w_{i_1,\ldots, i_k}z z^2 b_0^{j-1} & \mbox{if $j\geqslant3$}
\end{array}\right.
\!\!\!\sim \!
\left\{\begin{array}{ll}
w_{i_1,\ldots, i_k}z b_0^0 \in W_2 & \mbox{if $j=1,2$} \\
w_{i_1,\ldots, i_k}z b_0^{j-1} \in W_2 & \mbox{if $j\geqslant3$}.
\end{array}\right.
$$

Next, we treat the case $x=f_0$.
If $j=1$, then
$$
wf_0= w_{i_1,\ldots, i_k}zf_0\sim w_{i_1,\ldots, i_k}dz\sim_{R_d} dw_{i_1,\ldots, i_k}z\in W_3.
$$
On the other hand, for $j\geqslant2$, in $\PT(\Omega_{n-1})\zeta$ we have the equality
$(e_0b_0)^{n-3}e_0 b_0^{j-1} f_0 = (e_0b_0)^{n-3}e_0 b_0^{j-1}$,
whence $(e_0b_0)^{n-3}e_0 b_0^{j-1} f_0 \sim_{R_0} (e_0b_0)^{n-3}e_0 b_0^{j-1}$ and so
\begin{equation} \label{z2f}
z^2 b_0^{j-1} f_0 \sim z^2 b_0^{j-1},
\end{equation}
which implies that
$$
wf_0= w_{i_1,\ldots, i_k}zb_0^{j-1}f_0 \sim  w_{i_1,\ldots, i_k} z z^2 b_0^{j-1}f_0 \sim w_{i_1,\ldots, i_k} z z^2 b_0^{j-1} \sim w_{i_1,\ldots, i_k} z b_0^{j-1}\in W_2.
$$

To finish the case $w\in W_2$, we deal with $x=d$. Then,
$$
wd = w_{i_1,\ldots, i_k}zb_0^{j-1} d \sim_{R_d}  w_{i_1,\ldots, i_k} zd b_0^{j-1} \sim w_{i_1,\ldots, i_k} (f_0b_0)^{n-1}z b_0^{j-1} \sim_{R_0}  (f_0b_0)^{n-1}z b_0^{j-1} \in W_4.
$$

\smallskip

Thirdly, we suppose that $w\in W_3$. Take $1\leqslant i_1<\cdots<i_k\leqslant n-1$ and $1\leqslant k\leqslant n-1$ be such that $w=dw_{i_1,\ldots, i_k}z$.
If $x=d$, then
$$
wd=dw_{i_1,\ldots, i_k}zd \sim dw_{i_1,\ldots, i_k} (f_0b_0)^{n-1}z \sim_{R_0} d (f_0b_0)^{n-1}z \sim d (f_0b_0)^{n-1}\in W_1.
$$
On the other hand, if $x\in\{a_0,b_0,e_0,f_0\}$, then $(e_0b_0)^{n-3}e_0 f_0 x = (e_0b_0)^{n-3}e_0 f_0$ in $\PT(\Omega_{n-1})\zeta$.
Hence, we have
$(e_0b_0)^{n-3}e_0 f_0 x \sim_{R_0} (e_0b_0)^{n-3}e_0 f_0$ and so $z^2 f_0 x \sim z^2 f_0$, from which follows that
\begin{align*}
wx = dw_{i_1,\ldots, i_k}z x \sim_{R_d} w_{i_1,\ldots, i_k} dz x \sim w_{i_1,\ldots, i_k} zf_0 x \sim w_{i_1,\ldots, i_k} z z^2f_0 x
 \sim  w_{i_1,\ldots, i_k} z z^2f_0    \sim w_{i_1,\ldots, i_k} z f_0 \sim \qquad\quad \\ \sim w_{i_1,\ldots, i_k} d z \sim_{R_d}  dw_{i_1,\ldots, i_k} z =w \in W_3.
\end{align*}

\smallskip

Finally, we suppose that $w\in W_4$. Then, $w=(f_0b_0)^{n-1}zb_0^{j-1}$ for some $1\leqslant j\leqslant n-1$.

If $x=a_0$, in view of (\ref{z2}),  we have
\begin{align*}
wa_0 = (f_0b_0)^{n-1}zb_0^{j-1} a_0 \sim  (f_0b_0)^{n-1}z z^2b_0^{j-1} a_0
\sim
\left\{\begin{array}{ll}
 (f_0b_0)^{n-1}z z^2 b_0 & \mbox{if $j=1$} \\
 (f_0b_0)^{n-1}z z^2 b_0^0 & \mbox{if $j=2$} \\
 (f_0b_0)^{n-1}z z^2 b_0^{j-1} & \mbox{if $j\geqslant3$}
\end{array}\right.
\\  \sim
\left\{\begin{array}{ll}
 (f_0b_0)^{n-1}z b_0 \in W_4& \mbox{if $j=1$} \\
 (f_0b_0)^{n-1}z b_0^0 \in W_4& \mbox{if $j=2$} \\
 (f_0b_0)^{n-1}z b_0^{j-1} \in W_4& \mbox{if $j\geqslant3$}.
\end{array}\right.
\end{align*}

If $x=b_0$, then
$
wb_0= (f_0b_0)^{n-1}zb_0^{j-1} b_0 \sim_{R_0}
(f_0b_0)^{n-1}z b_0^{j'-1}\in W_4,
$
where $j'=j+1$, if $1\leqslant j\leqslant n-2$, and $j'=1$, if $j=n-1$.

If $x=e_0$, by using (\ref{z2e}), we get
\begin{align*}
we_0 = (f_0b_0)^{n-1}z b_0^{j-1} e_0  \sim  (f_0b_0)^{n-1}z z^2b_0^{j-1} e_0
 \sim
\left\{\begin{array}{ll}
(f_0b_0)^{n-1}z z^2 b_0^0 & \mbox{if $j=1,2$} \\
(f_0b_0)^{n-1}z z^2 b_0^{j-1} & \mbox{if $j\geqslant3$}
\end{array}\right.
\\ \sim
\left\{\begin{array}{ll}
(f_0b_0)^{n-1}z b_0^0 \in W_4& \mbox{if $j=1,2$} \\
(f_0b_0)^{n-1}z b_0^{j-1} \in W_4& \mbox{if $j\geqslant3$}.
\end{array}\right.
\end{align*}

Now, suppose that $x=f_0$. If $j=1$, then
$$
wf_0 = (f_0b_0)^{n-1}z f_0 \sim (f_0b_0)^{n-1} dz \sim_{R_d} d(f_0b_0)^{n-1} z \sim  d(f_0b_0)^{n-1} \in W_1.
$$
On the other hand, for $j\geqslant2$, by applying (\ref{z2f}), we obtain
$$
wf_0=(f_0b_0)^{n-1}zb_0^{j-1}f_0 \sim (f_0b_0)^{n-1}z z^2b_0^{j-1}f_0 \sim  (f_0b_0)^{n-1}z z^2b_0^{j-1} \sim (f_0b_0)^{n-1}z b_0^{j-1}=w\in W_4.
$$

To finishes the proof, it remains to consider $x=d$. Then, we have
$$
wd = (f_0b_0)^{n-1}zb_0^{j-1} d \sim_{R_d}  (f_0b_0)^{n-1} zd b_0^{j-1} \sim (f_0b_0)^{n-1} (f_0b_0)^{n-1}z b_0^{j-1}
\sim_{R_0} (f_0b_0)^{n-1}z b_0^{j-1}=w\in W_4,
$$
as required.
\end{proof}

At this point, notice that it is a routine matter to check that all relations of $R_\mathrm{s}$ are satisfied by the generators $a_0$, $b_0$, $e_0$, $f_0$, $d$ and $z$ of $\PsEnd(S_n)$.  Moreover, as already observed, we have $|W_\mathrm{s}|= |\PsEnd(S_n)|$. Thus, in view of Lemma \ref{sfun}, we may conclude the following result.

\begin{theorem}\label{presPs}
The presentation $\langle a_0,b_0,e_0,f_0,d,z\mid R_\mathrm{s}\rangle$ defines the monoid $\PsEnd(S_n)$.
\end{theorem}

Notice that, if we consider $R_0$ has being Popova's defining relations for $\PT(\Omega_{n-1})$ \cite{Popova:1961} based on Moore's presentation for
$\Sym(\Omega_{n-1})$ \cite{Moore:1897}, we have $|R_0|=n+16$ and so $|R_\mathrm{s}|=n+28$.

\section{A presentation for $\PswEnd(S_n)$}

Let $R_{\mathrm{sw}}$ be the set of relations $R_\mathrm{s}$ together with the following eight relations
\begin{align*}
b_0z_0=z_0,~
zz_0=z_0,~  z_0a_0=z_0,~ z_0b_0=z_0,~ z_0e_0=z_0,~ z_0f_0=z_0,~
dz_0=dz,~ z_0d=d(f_0b_0)^{n-1}
\end{align*}
on the alphabet $\{a_0,b_0,e_0,f_0,d,z,z_0\}$.
It is easy to verify that all relations of $R_{\mathrm{sw}}$ are satisfied by the generators $a_0$, $b_0$, $e_0$, $f_0$, $d$, $z$ and $z_0$ of $\PswEnd(S_n)$.

\smallskip

For each proper non-empty subset $A$ of $\Omega_{n-1}$, let $w_A$ be a word on the alphabet $\{a_0,b_0,f_0\}$ representing the partial identity $\id_{A\cup\{0\}}$. Let $w_{\Omega_{n-1}}$ be the empty word.
Consider the following sets of words on the alphabet $\{a_0,b_0,e_0,f_0,d,z,z_0\}$:
$$
W_5=\{w_Az_0\mid \emptyset \neq A\subseteq\Omega_{n-1}\}
$$
and
$$
W_6=\{w_Az_0zb_0^k\mid \emptyset \neq A\subseteq\Omega_{n-1},~0\leqslant k\leqslant n-2\}.
$$
Let us also consider a set of canonical forms for $\PsEnd(S_n)$ on the generators $a_0$, $b_0$, $e_0$, $f_0$, $d$ and $z$,
containing the empty word, for instance, $W_\mathrm{s}$. Let
$$
W_\mathrm{sw}=W_\mathrm{s}\cup W_5\cup W_6.
$$

Observe that,
$|W_\mathrm{sw}|=|W_\mathrm{s}|+|W_5|+|W_6|=(2n^{n-1} + n2^{n-1} -1) + (2^{n-1}-1) + (n-1)(2^{n-1}-1) =
2n^{n-1} + n2^n -n-1 =|\PswEnd(S_n)|$.

\begin{lemma}\label{z0}
$z_0^2\sim_{R_\mathrm{sw}}z_0$.
\end{lemma}
\begin{proof}
First, recall that $zd=(f_0b_0)^{n-1}z$ is a relation of $R_\mathrm{s}$,
whence $z_0z\sim z_0(f_0b_0)^{n-1}z\sim z_0zd$.
Secondly, notice that $zdz=\transf{0\\0}\in \langle a_0,b_0,f_0\rangle$ and so $zdz\sim_{R_\mathrm{s}}v$,
for some $v\in \{ a_0,b_0,f_0\}^*$.
Therefore,
$$
z_0^2\sim z_0zz_0 \sim z_0zdz_0\sim z_0zdz\sim z_0v\sim z_0,
$$
as required.
\end{proof}

\begin{lemma}\label{sfun2}
Let $w\in W_\mathrm{sw}$ and $x\in \{a_0,b_0,e_0,f_0,d,z,z_0\}$. Then, there exists $w'\in W_\mathrm{sw}$ such that $wx\sim_{R_\mathrm{sw}} w'$.
\end{lemma}
\begin{proof}
Let us denote $\sim_{R_\mathrm{sw}}$ simply by $\sim$. We will use Theorem \ref{presPs} several times without explicit mention.

\smallskip

First, let us suppose that $w\in W_\mathrm{s}$ and $x\in \{a_0,b_0,e_0,f_0,d,z\}$.
Then, there exists $w'\in W_\mathrm{s}$ such that $wx\sim_{R_\mathrm{s}}w'$.
Then, in particular, $w'\in W_\mathrm{sw}$ and $wx\sim w'$.

\smallskip

Secondly, suppose that $w\in W_\mathrm{s}$ and $x=z_0$.
Take $\alpha=wz\in\PsEnd(S_n)$.
Then $\im(\alpha)\subseteq\{0,1\}$. Let $A_0=0\alpha^{-1}$ and $A_1=1\alpha^{-1}$.
We will consider the various possibles cases for $\alpha$.

\smallskip

\noindent{\sc case} 1. $0\not\in\dom(\alpha)$. Hence, $\im(\alpha)=\{0\}$ or $\im(\alpha)\subseteq\Omega_{n-1}$
(in fact, moreover, $\im(\alpha)\subseteq\{1\}$).
Suppose that $\im(\alpha)\subseteq\Omega_{n-1}$. Then, $\alpha=\alpha d$ and so $wz\sim_{R_\mathrm{s}} wzd$,
whence $wz_0\sim wzz_0\sim wzdz_0\sim wzdz\in\{a_0,b_0,e_0,f_0,d,z\}^*$.
Thus, being  $w'\in W_\mathrm{s}$ such that $wzdz \sim_{R_\mathrm{s}} w'$, we have $wz_0\sim w'\in  W_\mathrm{sw}$.
On the other hand, suppose that $\im(\alpha)=\{0\}$. Then, $\im(\alpha z)=\{1\}$ and so $\alpha z =\alpha zd$,
from which follows that $wz^2\sim_{R_\mathrm{s}} wz^2d$,
whence $wz_0\sim wz^2z_0\sim wz^2dz_0\sim wz^2dz\in\{a_0,b_0,e_0,f_0,d,z\}^*$.
So, taking $w'\in W_\mathrm{s}$ such that $wz^2dz \sim_{R_\mathrm{s}} w'$, we also have $wz_0\sim w'\in  W_\mathrm{sw}$.

\smallskip

\noindent{\sc case} 2. $0\in\dom(\alpha)$ and $0\alpha=0$. Hence, $\Omega_{n-1}\alpha\subseteq\Omega_{n-1}$ and so
$A_0=\{0\}$.
If $A_1=\emptyset$, then $\alpha z =\transf{0\\1}=\alpha zd$ and so, just like in {\sc case} 1, there exists $w'\in W_\mathrm{s}\subseteq W_\mathrm{sw}$ such that $wz_0\sim w'$.
Therefore, suppose that $A_1\neq\emptyset$. Then, $\alpha=w_{A_1}z^2$ and so $wz\sim_{R_\mathrm{s}} w_{A_1}z^2$.
Hence, $wz_0\sim wzz_0\sim  w_{A_1}z^2z_0 \sim w_{A_1}z_0\in W_5$.

\smallskip

\noindent{\sc case} 3. $0\in\dom(\alpha)$ and $0\alpha\neq0$. Hence, $\Omega_{n-1}\alpha\subseteq\{0\}$ and so
$A_1=\{0\}$.
If $A_0=\emptyset$, then $\alpha =\transf{0\\1}=\alpha d$ and so, just as in {\sc case} 1, there exists $w'\in W_\mathrm{s}\subseteq W_\mathrm{sw}$ such that $wz_0\sim w'$.
Hence, suppose that $A_0\neq\emptyset$. Then, $\alpha=w_{A_0}z$ and so $wz\sim_{R_\mathrm{s}} w_{A_0}z$,
whence $wz_0\sim wzz_0\sim  w_{A_0}zz_0 \sim w_{A_0}z_0\in W_5$.

\smallskip

In third place, suppose $w\in W_5$, i.e. $w=w_Az_0$, for some non-empty subset $A$ of $\Omega_{n-1}$.

If $x\in \{a_0,b_0,e_0,f_0\}$, then $wx= w_Az_0x\sim w_Az_0\in W_5$.

If $x=d$, then $wx=w_Az_0d\sim w_A d(f_0b_0)^{n-1}\in \{a_0,b_0,e_0,f_0,d,z\}^*$ and so
$wx\sim  w'\in W_\mathrm{sw}$, with $w'\in W_\mathrm{s}$ such that $w_A d(f_0b_0)^{n-1} \sim_{R_\mathrm{s}} w'$.

If $x=z$, then $wx=w_Az_0z \in W_6$.

At last, if $x=z_0$, then $wx=w_Az_0z_0 \sim w_Az_0\in W_5$, by Lemma \ref{z0}.

\smallskip

It remains to consider that $w\in W_6$, i.e. $w=w_Az_0zb_0^k$, for some $0\leqslant k\leqslant n-2$ and some non-empty subset $A$ of $\Omega_{n-1}$.  Observe that $zb_0^k=\transf{0&1&\cdots&n-1\\k+1&0&\cdots&0}$.

If $x\in \{a_0,b_0,e_0,f_0\}$, then $zb_0^kx=\transf{0&1&\cdots&n-1\\\ell&0&\cdots&0}=zb_0^{\ell-1}$,
for some $\ell\in\Omega_{n-1}$, whence $zb_0^kx \sim_{R_\mathrm{s}} zb_0^{\ell-1}$
and so $wx = w_Az_0zb_0^kx \sim w_Az_0zb_0^{\ell-1}\in W_6$.

If $x=d$, then $wx = w_Az_0zb_0^kd \sim_{R_\mathrm{s}} w_Az_0zdb_0^k \sim_{R_\mathrm{s}} w_Az_0 (f_0b_0)^{n-1}z b_0^k \sim w_Az_0 z b_0^k\in W_6$.

If $x=z$, then $zb_0^kz=\transf{0&1&\cdots&n-1\\ 0&1&\cdots&1}\in \langle a_0,b_0,e_0\rangle$,
whence $zb_0^kz \sim_{R_\mathrm{s}} v$, for some $v\in \{ a_0,b_0,e_0\}^*$, and so
$wx = w_Az_0zb_0^k z \sim_{R_\mathrm{s}} w_Az_0 v \sim w_Az_0\in W_5$.

Finally, if $x=z_0$, then $wx= w_Az_0zb_0^k z_0  \sim w_A z_0z_0\sim w_Az_0\in W_5$, by Lemma \ref{z0}.

The proof is now complete.
\end{proof}

As already observed, all relations of $R_{\mathrm{sw}}$ are satisfied by the generators $a_0$, $b_0$, $e_0$, $f_0$, $d$, $z$ and $z_0$ of $\PswEnd(S_n)$ and $|W_\mathrm{sw}|=|\PswEnd(S_n)|$. This facts, together with the last lemma allows us to conclude the following result.

\begin{theorem}\label{presPsw}
The presentation $\langle a_0,b_0,e_0,f_0,d,z,z_0\mid R_{\mathrm{sw}}\rangle$ defines the monoid $\PswEnd(S_n)$.
\end{theorem}

Observe that, being $R_\mathrm{s}$ the set of defining relations for $\PsEnd(S_n)$ based on the presentations by Popova for $\PT(\Omega_{n-1})$  \cite{Popova:1961}  and by Moore for $\Sym(\Omega_{n-1})$ \cite{Moore:1897},
we have $|R_\mathrm{sw}|=n+36$.

\section{A presentation for $\PEnd(S_n)$}\label{pendsn}

Let $R_c$ be the set of relations $R_\mathrm{s}$ together with the relations
\begin{align*}
c^2=f_0d,~
cf_0=c,~ f_0c=f_0d,~ cd=f_0d,~  dc=c,~ \\
cb_0^{n-2}a_0b_0=b_0^{n-2}a_0b_0c,~ cb_0a_0=b_0a_0c,~\\
ca_0f_0a_0=a_0f_0a_0c,~
ca_0b_0^{n-2}a_0b_0e_0b_0^{n-2}a_0b_0a_0=a_0b_0^{n-2}a_0b_0e_0b_0^{n-2}a_0b_0a_0c,~ \\
czb_0^{n-2}=a_0(b_0e_0)^{n-3}c~
\end{align*}
on the alphabet $\{a_0,b_0,e_0,f_0,c,d,z\}$.
It is easy to verify that all relations of $R_c$ are satisfied by the generators $a_0$, $b_0$, $e_0$, $f_0$, $c$, $d$ and $z$ of $\PEnd(S_n)$.

\smallskip

Let us consider the following transformations of $\PT(\Omega_{n-1}^0)$:
$$
a'_0=b_0^{n-2}a_0b_0=\begin{pmatrix}
0&1&2&3&4&\cdots&n-1\\
0&1&3&2&4&\cdots&n-1
\end{pmatrix}, \quad
b'_0=b_0a_0=\begin{pmatrix}
0&1&2&\cdots&n-2&n-1\\
0&1&3&\cdots&n-1& 2
\end{pmatrix},
$$
$$
e'_0=a_0a'_0e_0a'_0a_0=\begin{pmatrix}
0&1&2&3&4&\cdots&n-1\\
0&1&2&2&4&\cdots&n-1
\end{pmatrix}
\quad\text{and}\quad
f'_0=a_0f_0a_0=\begin{pmatrix}
0&1&3&4&\cdots&n-1\\
0&1&3&4&\cdots&n-1
\end{pmatrix}.
$$
Also, let $\PT_{n-2}^{0,1}=\{\alpha\in\PT(\Omega_{n-1}^0)\mid 0\alpha=0, 1\alpha=1, \{2,\ldots,n-1\}\alpha\subseteq\{2,\ldots,n-1\}\}$.
Then, it is clear that $\PT_{n-2}^{0,1}$ is a submonoid of
$\PT(\Omega_{n-1})\zeta$ isomorphic to $\PT(\Omega_{n-2})$ admitting
$\{a'_0,b'_0,e'_0,f'_0\}$ as a generating set.

Let $R_1$ be the subset of $R_c$ formed by $R_0$ together with the following four relations:
\begin{align*}
cb_0^{n-2}a_0b_0=b_0^{n-2}a_0b_0c,~
cb_0a_0=b_0a_0c,~ ca_0f_0a_0=a_0f_0a_0c,~  \\
ca_0b_0^{n-2}a_0b_0e_0b_0^{n-2}a_0b_0a_0=a_0b_0^{n-2}a_0b_0e_0b_0^{n-2}a_0b_0a_0c.
\end{align*}

Since $\{a'_0,b'_0,e'_0,f'_0\}$ generates
$\PT_{n-2}^{0,1}$, the following lemma is easy to show.

\begin{lemma}\label{r1}
Let $w$ be a word on the alphabet $\{a_0,b_0,e_0,f_0\}$ such that (as transformation) $w\in \PT_{n-2}^{0,1}$. Then, $cw\sim_{R_1}wc$.
\end{lemma}

Next, we prove:

\begin{lemma}\label{cw}
Let $w$ be a word on the alphabet $\{a_0,b_0,e_0,f_0\}$ such that (as transformation)
$1\not\in\im(w)$. Then, there exists a word $w'$ on the alphabet $\{a_0,b_0,e_0,f_0\}$ such that $cw\sim_{R_c}w'c$ and
(as transformation) $w'\in \PT_{n-2}^{0,1}$.
\end{lemma}
\begin{proof}
Firstly, let us suppose that $1\not\in\dom(w)$. Then, $w=\transf{0&i_2&\cdots&i_k\\0&j_2&\cdots&j_k}$ for some
$2\leqslant i_2<\cdots<i_k\leqslant n-1$, $2\leqslant j_2,\ldots,j_k\leqslant n-1$ and $k\geqslant1$.
Hence, we have $w=\alpha f_0$ with $\alpha=\transf{0&1&i_2&\cdots&i_k\\0&1&j_2&\cdots&j_k}\in\PT_{n-2}^{0,1}$.
Next, let $w'$ be any word on the alphabet $\{a_0,b_0,e_0,f_0\}$ such that (as transformation)
$w'=\alpha$. Then, by Lemma \ref{r1}, $cw\sim_{R_0}cw'f_0\sim_{R_1}w'cf_0\sim_{R_c}w'c$.

Now, suppose that $1\in\dom(w)$. Then, $w=\transf{0&1&i_2&\cdots&i_k\\0&j&j_2&\cdots&j_k}$ for some $k\geqslant1$,
$2\leqslant i_2<\cdots<i_k\leqslant n-1$ and $2\leqslant j, j_2,\ldots,j_k\leqslant n-1$.
Hence, $w=\alpha\beta$ with $\alpha=\transf{0&1&i_2&\cdots&i_k\\0&1&j_2&\cdots&j_k}\in\PT_{n-2}^{0,1}$ and
$\beta=\transf{0&1&2&\cdots&n-1\\0&j&2&\cdots&n-1}$.
Let $w',w_1,w_2$ be any words on the alphabet $\{a_0,b_0,e_0,f_0\}$ such that (as transformations)
$w'=\alpha$, $w_1=\transf{0&1&2&\cdots&j-1&j&j+1&\cdots&n-1\\0&j&2&\cdots&j-1&1&j+1&\cdots&n-1}$
and $w_2=\transf{0&1&2&3&\cdots&j-1&j&j+1&\cdots&n-1\\0&1&j&3&\cdots&j-1&2&j+1&\cdots&n-1}\in\PT_{n-2}^{0,1}$.
Then, we have $\beta=w_2e_0w_2w_1$ and so, by Lemma \ref{r1},
$cw\sim_{R_0}cw'w_2e_0w_2w_1\sim_{R_1}w'w_2ce_0w_2w_1\sim_{R_0}w'w_2cf_0e_0w_2w_1
\sim_{R_0}w'w_2cf_0a_0w_2w_1
\sim_{R_c}w'w_2ca_0w_2w_1
\sim_{R_0}w'w_2cw_2\sim_{R_1}w'cw_2w_2\sim_{R_0}w'c$, as required.
\end{proof}

And, as a consequence of the previous lemma, we also have:

\begin{lemma}\label{cwf0}
Let $w$ be a word on the alphabet $\{a_0,b_0,e_0,f_0\}$.
Then, there exists a word $w'$ on the alphabet $\{a_0,b_0,e_0,f_0\}$ such that $cwf_0\sim_{R_c}w'c$ and
(as transformation) $w'\in \PT_{n-2}^{0,1}$.
\end{lemma}
\begin{proof}
Since (as transformation) $1\not\in\im(wf_0)$, by Lemma \ref{cw}, we immediately get $cwf_0\sim_{R_c}w'c$ for some
word $w'$ on the alphabet $\{a_0,b_0,e_0,f_0\}$ such that
(as transformation) $w'\in \PT_{n-2}^{0,1}$, as required.
\end{proof}

\smallskip

Now, for $1\leqslant k\leqslant n-1$, define
$$
B_{k}=\{\alpha \in \PsEnd(S_{n})\mid\mbox{$\{0,1,\ldots,k\}\subseteq \im(\alpha)$, $0\in\dom(\alpha)$, $0\alpha=0$ and $|\im(\alpha)|\geqslant3$}\}.
$$
Notice that, clearly, $B_{n-1} \subset B_{n-2} \subset \cdots \subset B_1$.
For $1\leqslant k\leqslant n-1$ , let
\begin{equation}\label{sk}
\sigma_{k}=\left(
\begin{array}{cccccccc}
0 & 1 & 2 & \cdots & k & k+1 & \cdots & n-1 \\
0 & k & 1 & \cdots & k-1 & k+1 & \cdots & n-1%
\end{array}%
\right)\in \PsEnd(S_{n}).
\end{equation}
Observe that, $\sigma_1$ is the identity on $\Omega_{n-1}^0$, $\sigma_2=a_0$ and $\sigma_{n-1}=b_0^{n-2}$.
Let $W_\mathrm{s}$ be a set of canonical forms for $\PsEnd(S_n)$ on the generators $a_{0}, b_{0}, e_{0}, f_{0}, d$ and $z$ containing the empty word.
For $\alpha \in \PsEnd(S_{n})$,
denote by $w_\alpha$ the word in $W_\mathrm{s}$ representing the transformation $\alpha$ and
consider the following set of words on the alphabet $\{a_{0},b_{0},e_{0},f_{0},d,z,c\}$:
$$
W_7 = \bigcup_{k=1}^{n-1}\left\{w_\alpha c w_{\sigma_k} \mid \alpha \in B_{k}\right\}.
$$
Notice that, $w_{\sigma_1}$ is the empty word. Let
$$
W_{c} = W_\mathrm{s} \cup W_7.
$$

\begin{lemma}
$| W_{c}| =| \PEnd(S_{n})|$.
\end{lemma}

\begin{proof}
Let $\alpha \in \PEnd(S_{n})\setminus \PsEnd(S_{n})$. Then, $0 \notin \dom\alpha$, $0 \in \im(\alpha)$ and $\Omega_{n-1} \cap \im(\alpha) \neq \emptyset$.
Since $|\im(\alpha)| \leqslant n-1$ and $0 \in \im(\alpha)$, there exists $\ell \in \Omega_{n-1}$ such that $\ell \notin \im(\alpha)$.
For $1\leqslant k\leqslant n-1$, we define
\[
C_{k}=\{\alpha \in \PEnd(S_{n})\setminus \PsEnd(S_{n}) \mid k = \min(\Omega_{n-1}\setminus\im(\alpha))\}.
\]
Then, $\PEnd(S_{n})\setminus \PsEnd(S_{n})=C_1\cup\cdots\cup C_{n-1}$ and $|\PEnd(S_{n})\setminus \PsEnd(S_{n})|=\sum_{k=1}^{n-1}|C_k|$.
Therefore, to conclude the proof of the lemma, it suffices to show that $| B_{k}| = | C_{k}|$, for all $1\leqslant k\leqslant n-1$.

Let $1\leqslant k\leqslant n-1$.

We define a mapping $\phi_{k} : B_{k} \rightarrow \PEnd(S_{n})$ by
$\alpha\phi_k=\alpha c \sigma_k$,
for $\alpha \in B_{k}$. It is routine to verify that $B_k\phi_{k}\subseteq C_{k}$.
On the other hand, since $c$ and $\sigma_k$ are injective mappings, so $c\sigma_k$ is also injective, and $\dom(c\sigma_k)=\Omega_{n-1}$, 
thus $\phi_k$ is an injective mapping. Hence, $|B_k|\leqslant |C_k|$.

Conversely, we define a mapping $\psi_k:C_k\rightarrow \PT(\Omega_{n-1}^0)$ by: for $\alpha\in C_k$, let $\dom(\alpha\psi_k) = \dom(\alpha) \cup \{0\}$ and
$$
x(\alpha\psi_k)=\left\{
\begin{array}{ll}
0 & \text{if } x=0 \\
x\alpha+1 & \text{if } 0 \leqslant x\alpha < k \\
x\alpha  & \text{if } k+1 \leqslant x\alpha,  \\
\end{array}%
\right.
$$
for $x \in \dom(\alpha)\cup\{0\}$
(observe that $\{0,1,\ldots,k-1\}\subseteq\im(\alpha)$ and $k\notin\im(\alpha)$).
It is clear that $C_k\psi_k\subseteq B_k$ (observe that, for $k=1$, we must have $\{2,\ldots,n-1\}\cap\im(\alpha)\neq\emptyset$,
whence $|\im(\alpha\psi_1)|\geqslant3$). On the other hand, it is routine to show that $\psi_k$ is injective. Thus, $|C_k|\leqslant |B_k|$ and so
$|B_k| = |C_k|$, as required.
\end{proof}

\begin{lemma}
Let $w\in W_{c}$ and let $x\in \{a_{0},b_{0},e_{0},f_{0},d,z,c\}$. Then, there exists $w'\in W_{c}$ such that $wx \sim_{R_{c}}w'$.
\end{lemma}

\begin{proof}
Let us denote $\sim_{R_{c}}$ simply by $\sim$. We will use Theorem \ref{presPs} several times without explicit mention.

\smallskip

First, let us suppose that $w\in W_\mathrm{s}$.  If $x\neq c$, then there exists $w'\in W_\mathrm{s} \subseteq W_{c}$ such that $wx \sim w'$, by Lemma \ref{sfun}.
So, suppose that $x = c$ and let $\alpha \in \PsEnd(S_{n})$ be such that $w=w_\alpha$.
If $\alpha \in B_{1}$, then $w c=w_\alpha c w_{\sigma_1} \in W_7 \subseteq W_{c}$.
Then, suppose that $\alpha \notin B_{1}$. It follows that $\alpha \notin B_{k}$, for all $1\leqslant k \leqslant n-1$, and thus
$1\notin \im(\alpha)$ or $0\notin\dom(\alpha)$ or $|\im(\alpha)|\leqslant2$ or $0\in\dom(\alpha)$, with $0\alpha\neq0$.

If $1\notin \im(\alpha)$, then $\alpha c = \alpha f_{0}c$ and so
$$
wc=w_\alpha c \sim w_\alpha f_{0}c \sim w_\alpha f_{0}d \sim w_{\alpha f_{0}d}\in W_\mathrm{s}.
$$

Next, suppose that $1\in \im(\alpha)$, $0\in\dom(\alpha)$ and $|\im(\alpha)|\leqslant2$. Then, $\im(\alpha)\subseteq\{0,1\}$ and so $\alpha=\alpha e_0(b_0e_0)^{n-3}$
and $\alpha e_0a_0=\alpha e_0a_0f_0$. Therefore,
\begin{align*}
wc=w_\alpha c\sim w_\alpha e_0(b_0e_0)^{n-3} c \sim w_\alpha  e_0a_0a_0(b_0e_0)^{n-3} c
\sim w_\alpha  e_0a_0f_0czb_0^{n-2}\sim  \qquad\quad \\
\sim w_\alpha  e_0a_0f_0dzb_0^{n-2} \sim w_{\alpha  e_0a_0f_0dzb_0^{n-2}}\in W_\mathrm{s}.
\end{align*}

Now, suppose that $1\in \im(\alpha)$ and $0\not\in\dom(\alpha)$.
Then, we have $\im(\alpha)\subseteq\Omega_{n-1}$.
Let $\alpha'$ be the extension of $\alpha$ such that $\dom(\alpha')=\dom(\alpha)\cup\{0\}$ and $0\alpha'=0$.
Clearly, $\alpha'\in\PsEnd(S_n)$ and $\alpha'd=\alpha$,
whence
$$
wc=w_\alpha c=w_{\alpha'd}c \sim w_{\alpha'} dc \sim w_{\alpha'} c.
$$
Therefore, if $\im(\alpha)=\{1\}$, then $wc\sim  w_{\alpha'} c\sim w'$, for some $w'\in W_\mathrm{s}$, by the previous case.
Otherwise, we get $\alpha'\in B_1$ and so $wc\sim w_{\alpha'} c \sim w_{\alpha'} c w_{\sigma_1}\in W_7$.

Finally, with regard to the remaining case, i.e. $1\in \im(\alpha)$, $0\in\dom(\alpha)$, $|\im(\alpha)|\geqslant3$ and $0\alpha\neq0$,
this is not possible in $\PsEnd(S_n)$.

\smallskip

Below, we suppose that $w \in W_7$. Let $1\leqslant k \leqslant n-1$ and $\alpha\in B_k$ be such that $w=w_\alpha c w_{\sigma_k}$.
The rest of the proof will be carried out considering each case for $x\in \{a_{0},b_{0},e_{0},f_{0},d,z,c\}$.

\smallskip

\noindent\textsc{case} $x=a_{0}$. If $k\geqslant 3$, then $\sigma_ka_0=b_{0}^{n-2}a_{0}b_{0}\sigma_k$, whence $w_{\sigma_k}a_0 \sim_{R_{\mathrm{s}}} b_{0}^{n-2}a_{0}b_{0}w_{\sigma_k}$ and so
$$
wa_0 = w_\alpha c w_{\sigma_k} a_{0} \sim w_\alpha c b_{0}^{n-2}a_{0}b_{0}w_{\sigma_k}
 \sim w_\alpha b_{0}^{n-2}a_{0}b_{0} c w_{\sigma_k} \sim w_{\alpha b_{0}^{n-2}a_{0}b_{0}} c w_{\sigma_k} \in W_7,
$$
since we also have $\alpha b_{0}^{n-2}a_{0}b_{0}\in B_k$.
If $k=2$, then $\sigma_k=a_0$ and, as $B_2\subset B_1$, we also have $\alpha\in B_1$, whence
$$
wa_0 = w_\alpha c w_{\sigma_k} a_{0} \sim w_\alpha c a_0 a_{0}\sim w_\alpha c w_{\sigma_1}\in W_7.
$$
Thus, suppose that $k=1$. In this case, we have $wa_0 = w_\alpha c w_{\sigma_1}a_{0} \sim w_\alpha c w_{\sigma_2}$.
If $2 \in \im(\alpha)$, then we get $\alpha \in B_2$ and so $w_\alpha c w_{\sigma_2} \in W_{7}$.
Next, suppose that $2 \notin \im(\alpha)$. Then, $\alpha a_{0}f_{0}a_{0}=\alpha$ and so $w_\alpha a_{0}f_{0}a_{0}\sim_{R_{\mathrm{s}}}w_\alpha$.
On the other hand, as $f_0a_0f_0=f_0a_0f_0a_0$, we get
\begin{align*}
wa_0=w_\alpha c a_{0} \sim w_\alpha a_0f_0a_0c a_{0} \sim w_\alpha c a_0f_0a_0 a_{0}
\sim w_\alpha c f_0a_0f_0\sim w_\alpha c f_0a_0f_0a_0 \sim w_\alpha c a_0f_0a_0 \sim \qquad\quad
\\ \sim w_\alpha a_0f_0a_0 c \sim w_\alpha c w_{\sigma_1} \in W_{7}.
\end{align*}

\smallskip

\noindent\textsc{case} $x=b_{0}$.
If $k=n-1$, then $\alpha\in B_1$ (since $B_{n-1}\subset B_1$) and
$$
wb_0 = w_\alpha cw_{\sigma_{n-1}}b_{0} \sim w_\alpha cb_{0}^{n-1} \sim w_{\alpha}c w_{\sigma_1} \in W_{7}.
$$
So, suppose that $1\leqslant k \leqslant n-2$. Then, $\sigma_kb_0 = b_{0}a_{0}\sigma_{k+1}$ and so
$$
wb_0 = w_\alpha cw_{\sigma_k}b_{0} \sim w_\alpha cb_{0}a_{0}w_{\sigma_{k+1}} \sim w_\alpha b_{0}a_{0}cw_{\sigma_{k+1}} \sim w_{\alpha b_{0}a_{0}}cw_{\sigma_{k+1}}.
$$
If $n-1 \in \im(\alpha)$, then $\{1,\ldots,k+1\}\subseteq \im(\alpha b_{0}a_{0})$, whence $\alpha b_0a_0\in B_{k+1}$ and so
$w_{\alpha b_{0}a_{0}}cw_{\sigma_{k+1}} \in W_7$.
Hence, suppose that $n-1 \notin \im(\alpha)$. Then, $2\notin\im(\alpha b_0a_0)$ and so $\alpha b_0a_0a_0f_0a_0=\alpha b_0a_0$.
On the other hand, $a_0f_0a_0\sigma_{k+1}=\sigma_{k+1}f_0$ and, by Lemma \ref{cwf0}
(notice that, as $\sigma_{k+1}$ is a permutation on $\Omega_{n-1}^0$, then $w_{\sigma_{k+1}}\in\{a_0,b_0\}^*$),
we have $cw_{\sigma_{k+1}}f_0\sim_{R_c}uc$,
for some $u\in\{a_0,b_0,e_0,f_0\}^*$ such that (as transformation) $u$ fixes $0$ and $1$.
Hence, we obtain
$$
wb_0\sim w_{\alpha} cb_{0}a_{0} w_{\sigma_{k+1}} \sim w_{\alpha} c b_{0}a_{0}a_0f_0a_0 w_{\sigma_{k+1}}
\sim w_{\alpha} c b_{0}a_{0} w_{\sigma_{k+1}}f_0  \sim w_{\alpha} b_{0}a_{0}  c w_{\sigma_{k+1}}f_0
\sim w_{\alpha} b_{0}a_{0} u c.
$$
In particular, we have $\alpha c\sigma_kb_0=\alpha b_0a_0uc$.
Since $\im(\alpha)\cap\{2,\ldots,n-1\}\neq\emptyset$, then $\im(\alpha c\sigma_kb_0)\cap\Omega_{n-1}\neq\emptyset$,
whence $\im(\alpha b_0a_0u)\cap\{2,\ldots,n-1\}\neq\emptyset$.
Moreover, $\alpha b_0a_0u$ also fixes $0$ and, since $1\in\im(\alpha)$ and both $b_0a_0$ and $u$ fix $1$, we also get $1\in\im(\alpha b_0a_0u)$.
Thus, $\alpha b_0a_0u\in B_1$.
Therefore, we have
$$
wb_0\sim w_{\alpha} b_{0}a_{0} u c\sim w_{\alpha b_{0}a_{0} u} c w_{\sigma_1} \in W_7.
$$

\smallskip

\noindent\textsc{case} $x=e_{0}$. First, observe that, $e_0=e_0b_0^{n-2}f_0b_0$ and,  by Lemma \ref{cwf0},
$cw_{\sigma_k e_0b_0^{n-2}}f_0\sim_{R_c}u c$, for some $u\in\{a_0,b_0,e_0,f_0\}^*$ such that (as transformation) $u$ fixes $0$ and $1$.
Hence, we have
$$
we_0=w_\alpha cw_{\sigma_k}e_{0} \sim w_\alpha cw_{\sigma_k}e_{0}b_0^{n-2}f_0b_0\sim
w_\alpha cw_{\sigma_ke_{0}b_0^{n-2}}f_0b_0\sim w_\alpha u cb_0
$$
and, in particular, $\alpha c\sigma_ke_0=\alpha ucb_0$.
As above, since $\im(\alpha)\cap\{2,\ldots,n-1\}\neq\emptyset$, then $\im(\alpha c\sigma_ke_0)\cap\Omega_{n-1}\neq\emptyset$, and so
$\im(\alpha u)\cap\{2,\ldots,n-1\}\neq\emptyset$.
In addition, as both $u$ and $\alpha$ fix $0$, $1\in\im(\alpha)$ and $u$ fixes $1$, we obtain $\alpha u\in B_1$.
Then, by the previous case, we have $w_{\alpha u} cw_{\sigma_1}b_0 \sim w'$, for some $w'\in W_7$.
Therefore,
$$
we_0 \sim w_\alpha u cb_0 \sim w_{\alpha u} cw_{\sigma_1}b_0 \sim w' \in W_7.
$$

\smallskip

\noindent\textsc{case} $x=f_{0}$.
First, suppose that $\im(\alpha)\cap\{3,\ldots,n-1\}\neq\emptyset$.
Then, $\im(\alpha c\sigma_kf_0)\cap\Omega_{n-1}\neq\emptyset$.
By Lemma \ref{cwf0}, take $u\in\{a_0,b_0,e_0,f_0\}^*$ such that (as transformation) $u$ fixes $0$ and $1$ and
$cw_{\sigma_k}f_0\sim_{R_c}u c$.
Hence, $\alpha c\sigma_kf_0=\alpha uc$, $0,1\in\im(\alpha u)$ and $\im(\alpha u)\cap\{2,\ldots,n-1\}\neq\emptyset$.
Thus, $\alpha u\in B_1$ and so
$$
wf_0=w_\alpha cw_{\sigma_k}f_{0} \sim w_\alpha u c \sim w_{\alpha u} cw_{\sigma_1} \in W_7.
$$
Conversely, suppose that $\im(\alpha)=\{0,1,2\}$. Then $k\in\{1,2\}$. If $k=1$, then
$$
wf_0=w_\alpha cf_{0} \sim w_\alpha c =w_\alpha cw_{\sigma_1} \in W_7.
$$
If $k=2$, since  $\alpha a_0f_0a_0 = \alpha a_0f_0a_0 (b_0e_0)^{n-3}$ and $f_0a_0f_0=f_0a_0f_0a_0$, we have
\begin{align*}
wf_0=w_\alpha cw_{\sigma_2}f_{0}\sim w_\alpha ca_0f_{0}\sim w_\alpha cf_0a_0f_{0}
\sim w_\alpha cf_0a_0f_{0} a_0 \sim w_\alpha c a_0f_{0} a_0  \sim w_\alpha a_0f_{0} a_0 c \sim \quad\qquad\\
\sim w_\alpha  a_0f_0a_0 (b_0e_0)^{n-3} c
\sim w_\alpha  a_0f_0 c zb_0^{n-2}
\sim w_\alpha  a_0f_0 d zb_0^{n-2} \sim w_{\alpha  a_0f_0 d zb_0^{n-2}}\in W_\mathrm{s}.
\end{align*}

\smallskip

\noindent\textsc{case} $x=c$. As above, by Lemma \ref{cwf0}, take $u\in\{a_0,b_0,e_0,f_0\}^*$ such that $cw_{\sigma_k}f_0\sim_{R_c}u c$.
Then, as $f_0\sigma_k=\sigma_kf_0$, we have
$$
wc=w_\alpha cw_{\sigma_k}c \sim w_\alpha cf_0w_{\sigma_k}c \sim w_\alpha cw_{\sigma_k}f_0c
\sim w_\alpha uc^2  \sim w_\alpha u f_0d \sim w_{\alpha u f_0d} \in W_\mathrm{s}.
$$

\smallskip

\noindent\textsc{case} $x=d$. Then, as $d\sigma_k=\sigma_kd$, we have
$$
wd=w_\alpha cw_{\sigma_k}d \sim w_\alpha cdw_{\sigma_k} \sim w_\alpha f_0dw_{\sigma_k}\sim w_{\alpha f_0dw_{\sigma_k}}\in W_\mathrm{s}.
$$

\smallskip

\noindent\textsc{case} $x=z$. In this last case, we begin by observing that $\sigma_k z=z$. Next, notice that $\alpha a_{0}(b_{0}e_{0})^{n-3}$ fixes $0$ and,
as  $\im(\alpha)\cap\{2,\ldots,n-1\}\neq\emptyset$, we have $\im(\alpha a_{0}(b_{0}e_{0})^{n-3})=\{0,1,n-1\}$,
whence $\alpha a_{0}(b_{0}e_{0})^{n-3}\in B_1$.
Therefore,
$$
wz=w_\alpha cw_{\sigma_k} z \sim w_\alpha cz \sim w_\alpha czb_{0}^{n-1} \sim w_\alpha a_{0}(b_{0}e_{0})^{n-3}cb_{0} \sim w_{\alpha a_{0}(b_{0}e_{0})^{n-3}}cw_{\sigma_1}b_{0} \sim w',
$$
for some $w'\in W_7$, by the \textsc{case} $x=b_0$ with $k=1$.

The proof is now complete.
\end{proof}

Since all relations of $R_c$ are satisfied by the generators $a_0$, $b_0$, $e_0$, $f_0$, $d$, $c$ and $z$ of $\PEnd(S_n)$ and $|W_c|=|\PEnd(S_n)|$,
taking into account the last lemma, we can conclude the following theorem.

\begin{theorem}\label{presP}
The presentation $\langle a_0,b_0,e_0,f_0,c,d,z\mid R_c\rangle$ defines the monoid $\PEnd(S_n)$.
\end{theorem}

Observe that, considering again the set $R_\mathrm{s}$ of defining relations for $\PsEnd(S_n)$ based on the presentations by Popova for $\PT(\Omega_{n-1})$  \cite{Popova:1961} and by Moore  for $\Sym(\Omega_{n-1})$ \cite{Moore:1897}, we have $|R_c|=n+38$.

\section{A presentation for $\PwEnd(S_n)$}

Let $R_\mathrm{w}$ be the set of relations $R_\mathrm{s}$ together with the relations
\begin{align*}
c_0^2=c_0,~
e_0c_0=c_0a_0c_0, ~
c_0f_0=c_0,~ f_0c_0=f_0,~ c_0d=f_0d,~ \\
c_0b_0^{n-2}a_0b_0=b_0^{n-2}a_0b_0c_0,~ c_0b_0a_0=b_0a_0c_0,~\\
c_0a_0f_0a_0=a_0f_0a_0c_0,~
c_0a_0b_0^{n-2}a_0b_0e_0b_0^{n-2}a_0b_0a_0=a_0b_0^{n-2}a_0b_0e_0b_0^{n-2}a_0b_0a_0c_0,~ \\
c_0zc_0=zc_0,~
z^2c_0=zc_0,~
dc_0zb_0^{n-2}=a_0(b_0e_0)^{n-3}dc_0
\end{align*}
on the alphabet $\{a_0,b_0,e_0,f_0,c_0,d,z\}$.
It is easy to verify that all relations of $R_\mathrm{w}$ are satisfied by the generators $a_0$, $b_0$, $e_0$, $f_0$, $c_0$, $d$ and $z$ of $\PwEnd(S_n)$.

\smallskip

With a similar argument to Lemma \ref{r1}, we can prove the following lemma.

\begin{lemma}\label{r1w}
Let $w$ be a word on the alphabet $\{a_0,b_0,e_0,f_0\}$ such that (as transformation) $w\in \PT_{n-2}^{0,1}$. Then, $c_0w\sim_{R_\mathrm{w}} wc_0$.
\end{lemma}

And, similarly to Lemma \ref{cwf0}, we can also prove:

\begin{lemma}\label{PwEnd}
Let $w$ be a word on the alphabet $\{a_{0},b_{0},e_{0},f_{0}\}$. Then, there exists a word $w'$ on the alphabet $\{a_{0},b_{0},e_{0},f_{0}\}$ such that $c_0wf_{0}\sim_{R_\mathrm{w}}w'c_0$ and (as transformation) $w'\in \PT_{n-2}^{0,1}$.
\end{lemma}

\smallskip

Next, let us define
$$
D=\left\{ \transf{0&A&B\\0&1&2}\in\PT(\Omega_{n-1}^0)\mid A,B\subseteq\Omega_{n-1}, A\neq\emptyset\right\}
$$
and
$$
D_{k}=\{\alpha\! \in\! \PsEnd(S_{n}) \mid \mbox{$\{1,\ldots,k\}\!\subseteq\! \im(\alpha)$ and either $0\!\in\! \dom(\alpha)$ and $0\alpha = 0$
or  $0\!\notin\! \dom(\alpha)$  and $|\im(\alpha)| \geqslant 2$}\},
$$
for $1 \leqslant k \leqslant n-1$.
Then, obviously, $D_{n-1} \subset D_{n-2} \subset \cdots \subset D_1$. It is also clear that $D\subset D_1$.

As in Section \ref{pendsn}, let us take a set $W_\mathrm{s}$ of canonical forms for $\PsEnd(S_n)$ on the generators $a_{0}, b_{0}, e_{0}, f_{0}, d$ and $z$ containing the empty word and, for $\alpha \in \PsEnd(S_{n})$, denote by $w_\alpha$ the word in $W_\mathrm{s}$ representing the transformation $\alpha$.
So, consider the following sets of words on the alphabet $\{a_{0},b_{0},e_{0},f_{0},d,z,c_{0}\}$:
$$
W_{8} = \bigcup_{k=1}^{n-1}\left\{w_\alpha c_0 w_{\sigma_k} \mid \alpha \in D_{k}\right\},
$$
with $\sigma_1,\ldots,\sigma_k$ defined as in (\ref{sk}) and
$$
W_{9} = \{w_\alpha c_0zb_{0}^{\ell} \mid \alpha \in D,~ 0\leqslant \ell \leqslant n-2\}.
$$
Recall that, $\sigma_1$ is the identity on $\Omega_{n-1}^0$ (whence $w_{\sigma_1}$ is the empty word), $\sigma_2=a_0$ and $\sigma_{n-1}=b_0^{n-2}$.
Let
$$
W_\mathrm{w} = W_\mathrm{s} \cup W_8 \cup W_9.
$$

\begin{lemma}\label{PwEnd2}
$|W_\mathrm{w}| =|\PwEnd(S_{n})|$.
\end{lemma}

\begin{proof}
Let $\alpha \in \PwEnd(S_{n})\setminus \PsEnd(S_{n})$. Then, either
$0\not\in\dom(\alpha)$ and $0\in\im(\alpha)$
or $0\in\dom(\alpha)$ and $0\alpha\in\Omega_{n-1}\alpha$.
So, let us define
$$
\Gamma_0 = \{\alpha \in \PwEnd(S_{n})\setminus \PsEnd(S_{n})\mid \mbox{$0 \in \dom(\alpha)$ and $0\alpha \neq 0$}\}
$$
and, for $1\leqslant k\leqslant n-1$,
$$
\Gamma_{k} = \{\alpha\! \in\! \PwEnd(S_{n})\setminus \PsEnd(S_{n}) \mid \mbox{$k = \min(\Omega_{n-1}\! \setminus\! \Omega_{n-1}\alpha)$
and either $0\! \notin\! \dom(\alpha)$ or $0\! \in\! \dom(\alpha)$ and $0\alpha = 0$}\}.
$$

Therefore, as $\Omega_{n-1}\setminus\Omega_{n-1}\alpha\neq\emptyset$,  it is clear that
$\PwEnd(S_{n})\setminus \PsEnd(S_{n}) = \Gamma_0 \cup \Gamma_1 \cup \cdots \cup \Gamma_{n-1}$
and $|\PwEnd(S_{n})\setminus \PsEnd(S_{n})| = \sum_{k=0}^{n-1}|\Gamma_k|$.

\smallskip

First, observe that, we can also write
$$
\Gamma_0=\left\{ \transf{0&A&B\\\ell&\ell&0}\in\PT(\Omega_{n-1}^0)\mid A,B\subseteq\Omega_{n-1}, A\neq\emptyset, 1\leqslant\ell\leqslant n-1\right\}
$$
and so, clearly, $|\Gamma_0| =(n-1)|D|$.

\smallskip

Next, we are going to show that
$| \Gamma_{k}| =| D_{k}|$ for all $1\leqslant k\leqslant n-1$.

Let $1\leqslant k\leqslant n-1$. We define a mapping $\phi_{k}: D_{k}\rightarrow \PwEnd(S_{n})$ by
$\alpha\phi_{k}= \alpha c_{0}\sigma_k,$
for $\alpha \in D_{k}$.
It is a routine matter to verify that $D_k\phi_{k}\subseteq \Gamma_{k}$.
On the other hand, observing that $\Omega_{n-1}\subset\dom(c_0\sigma_k)$, $(c_0\sigma_k)|_{\Omega_{n-1}}$ is injective and,
for $\alpha\in D_k$, we have $\Omega_{n-1}\alpha\subseteq \Omega_{n-1}$ (notice that, as $1\in\im(\alpha)$, if $0\notin\dom(\alpha)$, then $\Omega_{n-1}\alpha=\im(\alpha)\subseteq\Omega_{n-1}$) and $\dom(\alpha\phi_k)=\dom(\alpha)$,
we can deduce that $\phi_k$ is an injective mapping
and so $| D_{k}| \leqslant | \Gamma_{k}|$.
Conversely, we define a mapping $\psi_k: \Gamma_k \rightarrow \PT(\Omega_{n-1}^0)$ by:
for $\alpha \in \Gamma_k$, let $\dom(\alpha\psi_k)=\dom(\alpha)$ and
$$
x(\alpha\psi_k)=\left\{
\begin{array}{ll}
1 & \mbox{if $x\neq0$ and $x\alpha=0$} \\
x\alpha\sigma_k^{-1} & \mbox{otherwise,}
\end{array}%
\right.
$$
for $x\in \dom(\alpha)$.
Observe that, for $\alpha \in \Gamma_k$, we have $\{1,\ldots,k-1\}\subseteq\im(\alpha)$, whence $\{1,\ldots,k\}\subseteq\im(\alpha\psi_k)$.
Moreover, if $0\in\dom(\alpha)$, then $0\in\dom(\alpha\psi_k)$ and, clearly, $0(\alpha\psi_k)=0$.
On the other hand, if $0\notin\dom(\alpha)$, then there exist $x,y\in\dom(\alpha)$ such that $x\alpha=0$ and $y\alpha\in\Omega_{n-1}\setminus\{k\}$, whence $x(\alpha\psi_k)=1$ and
$y(\alpha\psi_k)\neq1$, which implies that $|\im(\alpha\psi_k)|\geqslant2$.
Hence, $\Gamma_k\psi_k \subseteq D_{k}$.
Furthermore, it is easy to show that $\psi_k$ is injective. Thus, $|\Gamma_k| \leqslant |D_k|$ and so $|\Gamma_k| = |D_k|$.

\smallskip

Thus, we have
$
|W_\mathrm{w}|=|W_\mathrm{s}|+|W_8|+|W_9|=|\PsEnd(S_n)|+\sum_{k=1}^{n-1}|D_k|+(n-1)|D|=|\PsEnd(S_n)|+\sum_{k=0}^{n-1}|\Gamma_k|=
|\PsEnd(S_n)|+|\PwEnd(S_{n})\setminus \PsEnd(S_{n})|=|\PwEnd(S_n)|,
$
as required.
\end{proof}

\begin{lemma}\label{PwEnd1}
Let $w\in W_\mathrm{w}$ and let $x\in \{a_{0},b_{0},e_{0},f_{0},d,z,c_{0}\}$.
Then, there exists $w'\in W_\mathrm{w}$ such that $wx \sim_{R_\mathrm{w}}w'$.
\end{lemma}

\begin{proof}
In this proof, we denote $\sim_{R_\mathrm{w}}$ simply by $\sim$.

Let $w\in W_\mathrm{w}$.

First, let us suppose that $w\in W_\mathrm{s}$. If $x\neq c_{0}$, then there exists $w'\in W_\mathrm{s} \subseteq W_\mathrm{w}$ such that $wx \sim w'$, by Lemma 3.3.
So, suppose that $x=c_0$ and let $\alpha \in \PsEnd(S_{n})$ be such that $w=w_\alpha$.
If $\alpha \in D_1$, then $wc_{0} = w_\alpha c_0 = w_\alpha c_0 w_{\sigma_1} \in W_8 \subseteq W_\mathrm{w}$.
Hence, suppose that $\alpha \notin D_1$.
It follows that $\alpha \notin D_k$, for $1 \leqslant k \leqslant n-1$, and thus either
$1 \notin \im(\alpha)$ or $0\in \dom(\alpha)$ and $0\alpha \neq 0$ or $0\notin \dom(\alpha)$ and $|\im(\alpha)| \leqslant 1$.

If $1 \notin \im(\alpha)$, then $\alpha c_{0} = \alpha f_{0}c_{0}$ and thus, we have
$$
wc_{0} = w_\alpha c_{0} \sim w_\alpha f_{0}c_{0} \sim w_\alpha f_{0} \sim w_{\alpha f_{0}} \in W_\mathrm{s}.
$$

Therefore, suppose that $1 \in \im(\alpha)$.
Then, $0\in \dom(\alpha)$ and $0\alpha \neq 0$ or $0\notin \dom(\alpha)$ and $|\im(\alpha)| \leqslant 1$.
It follows that $\{1\}\subseteq \im(\alpha) \subseteq \{0,1\}$ and $\alpha = \alpha z^2$,
whence
$$
wc_{0} = w_\alpha c_{0} \sim w_{\alpha z^2} c_{0} \sim w_\alpha z^2 c_{0} \sim w_\alpha zc_{0}\sim w_{\alpha z}c_{0}.
$$
If $\im(\alpha)=\{0,1\}$, then $\alpha z \in D_1$ and so $wc_{0} \sim w_{\alpha z}c_{0} w_{\sigma_1}\in W_8$.
On the other hand, if $\im(\alpha)=\{1\}$, then $\alpha z=\alpha zf_0$ and so
\begin{equation}\label{w=1}
wc_{0} \sim w_{\alpha z}c_{0}\sim  w_{\alpha z f_0}c_{0} \sim  w_{\alpha z} f_0c_{0}\sim  w_{\alpha z} f_{0} \sim  w_{\alpha z f_{0}}\in W_\mathrm{s}.
\end{equation}

\smallskip

Next, we suppose that $w \in W_8$ and consider each case for $x \in \{a_{0},b_{0},e_{0},f_{0},d,z,c_{0}\}$.
Let $1\leqslant k\leqslant n-1$ and let $\alpha\in D_k$ be such that $w=w_\alpha c_0 w_{\sigma_k}$.

\smallskip

\noindent\textsc{case} $x=a_{0}$.

Suppose that $k=1$.
If $2\in\im(\alpha)$, then $\alpha\in D_2$ and so
$$
wa_0=w_\alpha c_0 w_{\sigma_1} a_0=w_\alpha c_0 a_0\sim w_\alpha c_0 w_{\sigma_2}\in W_8.
$$
So, suppose that $2\notin \im(\alpha)$. Hence, since $\alpha = \alpha a_{0}f_{0}a_{0}$ and $a_{0}f_{0}a_{0}f_0 = f_{0}a_{0}f_{0}$, we get
\begin{align*}
wa_0=w_\alpha c_0 a_0 \sim w_{\alpha a_0f_0a_0}c_0 a_0 \sim w_{\alpha} a_0f_0a_0c_0 a_0 \sim w_{\alpha} c_0a_0f_0a_0 a_0
\sim w_{\alpha} c_0f_0a_0f_0 \sim\qquad\\
\sim w_{\alpha} c_0a_0f_0a_0f_0 \sim w_{\alpha} a_0f_0a_0c_0f_0\sim w_{\alpha a_0f_0a_0}c_0 \sim w_\alpha c_0 w_{\sigma_1}\in W_8.
\end{align*}

If $k=2$, then $w=w_\alpha c_0 w_{\sigma_2}\sim w_\alpha c_0 a_0$ and, since $D_2 \subset D_1$, we get
$
wa_0 \sim w_\alpha c_0 a_0 a_0 \sim w_\alpha c_{0}w_{\sigma_1} \in W_8.
$

If $3 \leqslant k \leqslant n-1$, then
$\sigma_ka_0=b_{0}^{n-2}a_{0}b_{0}\sigma_k$,
whence $w_{\sigma_k}a_0 \sim b_{0}^{n-2}a_{0}b_{0}w_{\sigma_k}$ and so
$$
wa_0=w_\alpha c_{0}w_{\sigma_k}a_{0} \sim w_\alpha c_{0}b_{0}^{n-2}a_{0}b_{0}w_{\sigma_k} \sim w_\alpha b_{0}^{n-2}a_{0}b_{0}c_{0}w_{\sigma_k} \sim
w_{\alpha b_{0}^{n-2}a_{0}b_{0}}c_{0}w_{\sigma_k} \in W_8,$$
since $\alpha b_{0}^{n-2}a_{0}b_{0} \in D_k$.

\smallskip

\noindent\textsc{case} $x=b_0$.

If $1\leqslant k \leqslant n-2$, then $\sigma_k b_0 = b_{0}a_{0}\sigma_{k+1}$ and so
$$
wb_0=w_\alpha c_0 w_{\sigma_k}b_0 \sim w_\alpha c_0 b_{0}a_{0}w_{\sigma_{k+1}}
\sim w_\alpha b_{0}a_{0}c_{0}w_{\sigma_{k+1}} \sim w_{\alpha b_{0}a_{0}}c_{0}w_{\sigma_{k+1}}.
$$
If $n-1\in \im(\alpha)$, then $\alpha b_{0}a_{0} \in D_{k+1}$, whence $wb_0 \sim w_{\alpha b_{0}a_{0}}c_{0}w_{\sigma_{k+1}} \in W_8$.
So, suppose that $n-1\notin \im(\alpha)$.
Then, $2 \notin \im(\alpha b_{0}a_{0})$ and so $\alpha b_{0}a_{0} = \alpha b_{0}a_{0}(a_{0}f_{0}a_{0})$.
Hence, as $f_0a_0\sigma_{k+1}=f_0a_0\sigma_{k+1}f_0$ and $w_{\sigma_{k+1}}$ is a word on the alphabet $\{a_0,b_0\}$,
we have
\begin{align*}
wb_0\sim w_\alpha b_{0}a_{0}c_{0}w_{\sigma_{k+1}} \sim w_\alpha b_{0}a_{0}a_{0}f_{0}a_{0}c_{0}w_{\sigma_{k+1}} \sim w_\alpha b_{0}a_{0}c_{0}a_{0}f_{0}a_0 w_{\sigma_{k+1}} \sim\qquad\\
\sim w_\alpha b_{0}a_{0}c_{0}f_{0}a_{0}f_{0}a_0w_{\sigma_{k+1}}
\sim w_\alpha b_{0}a_{0}c_{0}f_{0}a_{0}f_{0}a_0w_{\sigma_{k+1}}f_{0} \sim w_\alpha b_{0}a_{0}uc_{0},
\end{align*}
for some word $u$ on the alphabet $\{a_0,b_0,e_0,f_0\}$ such that (as transformation) $u\in \PT_{n-2}^{0,1}$, by Lemma \ref{PwEnd}.
Hence, we get $\alpha b_{0}a_{0}u \in D_1$ and so
$wb_0 \sim w_\alpha b_{0}a_{0}uc_{0} \sim w_{\alpha b_{0}a_{0}u}c_{0}w_{\sigma_1} \in W_8$.

If $k = n-1$, then $\sigma_{n-1} = b_0^{n-2}$ and, as $D_{n-1} \subset D_1$, we obtain
$$
wb_0 = w_\alpha c_0 w_{\sigma_{n-1}}b_0 \sim w_\alpha c_0b_0^{n-2}b_0 \sim w_\alpha c_0b_0^{n-1}  \sim w_\alpha c_0 w_{\sigma_1} \in W_{8}.
$$

\smallskip

\noindent\textsc{case} $x=e_0$.

Since $e_0=e_0b_0^{n-2}f_0b_0$, we have
$$
we_0=w_\alpha c_{0}w_{\sigma_k}e_{0} \sim w_\alpha c_{0}w_{\sigma_k}e_{0}b_{0}^{n-2}f_{0}b_{0}
\sim w_\alpha uc_{0}b_{0}\sim w_{\alpha u}c_0w_{\sigma_1}b_0,
$$
for some word $u$ on the alphabet $\{a_0,b_0,e_0,f_0\}$ such that (as transformation) $u\in \PT_{n-2}^{0,1}$, by Lemma \ref{PwEnd}.
Moreover, clearly, $\alpha u\in D_1$.
Thus, $w_{\alpha u}c_0w_{\sigma_1}\in W_8$ and,
by the previous case, there exists $w' \in W_8$ such that $we_0\sim w_{\alpha u}c_0w_{\sigma_1}b_0 \sim w'$.

\smallskip

\noindent\textsc{case} $x=f_0$.

By Lemma \ref{PwEnd},
there exists a word $u$ on the alphabet $\{a_0,b_0,e_0,f_0\}$ such that
$c_0w_{\sigma_k}f_0\sim uc_0$ and
(as transformation) $u\in \PT_{n-2}^{0,1}$,
whence
$$
wf_0=w_\alpha c_{0}w_{\sigma_k}f_{0}\sim w_\alpha u c_0\sim w_{\alpha u} c_0 w_{\sigma_1}\in W_8,
$$
since $\alpha u \in D_1$.

\smallskip

\noindent\textsc{case} $x=d$.

In this case, we have
$
wd=w_\alpha c_0w_{\sigma_k}d \sim w_\alpha c_{0}d w_{\sigma_k} \sim w_\alpha f_{0}dw_{\sigma_k} \sim w_{\alpha f_{0}d\sigma_k} \in W_\mathrm{s}.
$

\smallskip

\noindent\textsc{case} $x=z$.

Observe that, as $\sigma_k z=z$, we obtain $wz=w_\alpha c_0w_{\sigma_k}z\sim w_\alpha c_0z$.

First, suppose that $0\notin \im(\alpha)$. Then,
$$
wz\sim w_\alpha c_{0}z \sim w_\alpha dc_{0}z \sim w_\alpha dc_{0}zb_{0}^{n-2}b_{0} \sim w_\alpha a_{0}(b_{0}e_{0})^{n-3}dc_{0}b_{0}
\sim w_{\alpha a_{0}(b_{0}e_{0})^{n-3}d}c_{0}w_{\sigma_1}b_{0}.
$$
Next, observe that, as $0\notin \im(\alpha)$, we must have $0\notin\dom(\alpha)$ and $\{1,t\}\subseteq\im(\alpha)$, for some $2\leqslant t\leqslant n-1$,
from which follows that $0\not\in\dom(\alpha a_{0}(b_{0}e_{0})^{n-3}d)$ and $\im(\alpha a_{0}(b_{0}e_{0})^{n-3}d)=\{1,n-1\}$,
whence $\alpha a_{0}(b_{0}e_{0})^{n-3}d\in D_1$.
Therefore, by the case $x=b_0$, there exists $w'\in W_8$ such that $wz\sim w_{\alpha a_{0}(b_{0}e_{0})^{n-3}d}c_{0}w_{\sigma_1}b_{0}\sim w'$.

Now, suppose that $0\in\im(\alpha)$. Since $\alpha\in D_k$, we must have $0\in\dom(\alpha)$ and $0\alpha=0$. Furthermore, $\{0,1\}\subseteq\im(\alpha)$.
Let $\beta=\transf{0&1&2&\cdots&n-1\\0&1&2&\cdots&2}\in\PT_{n-2}^{0,1}$.
Then, $\alpha\beta\in D$, $\beta z=z$ and, by Lemma \ref{r1w}, $c_0w_\beta \sim w_\beta c_0$.
Hence,
$$
wz\sim w_\alpha c_0 z\sim w_\alpha c_0 w_\beta z\sim w_\alpha  w_\beta c_0 z\sim w_{\alpha\beta} c_0 z b_0^0 \in W_9.
$$

\smallskip

\noindent\textsc{case} $x=c_{0}$.

If $k=1$, then $wc_0=w_\alpha c_0 w_{\sigma_1}c_0 \sim w_\alpha c_0c_0 \sim w_\alpha c_0 \sim w_\alpha c_0 w_{\sigma_1} \in W_{8}$.

So, suppose that $2 \leqslant k \leqslant n-1$. Then, we have $f_0\sigma_k = f_0\sigma_k(b_0a_0)^{n-k}a_0f_0a_0(b_0a_0)^{k-2}$ and,
by Lemma \ref{PwEnd},
there exists a word $u$ on the alphabet $\{a_0,b_0,e_0,f_0\}$ such that
$c_0f_0w_{\sigma_k}(b_0a_0)^{n-k}a_0f_0\sim uc_0$ and
(as transformation) $u\in \PT_{n-2}^{0,1}$.
Hence,
\begin{align*}
wc_0 = w_\alpha c_0 w_{\sigma_k}c_0 \sim w_\alpha c_0 f_0 w_{\sigma_k}c_0 \sim
w_\alpha c_0 f_0w_{\sigma_k}(b_0a_0)^{n-k}a_0f_0a_0(b_0a_0)^{k-2} c_0 \sim
w_\alpha u c_0 a_0(b_0a_0)^{k-2} c_0 \sim \qquad \\ \sim
w_\alpha u c_0 a_0c_0(b_0a_0)^{k-2}  \sim
w_\alpha u e_0c_0(b_0a_0)^{k-2}  \sim w_\alpha u e_0(b_0a_0)^{k-2} c_0  \sim w_{\alpha u e_0(b_0a_0)^{k-2}} c_0.
\end{align*}
Since $1u e_0(b_0a_0)^{k-2}=1$ and $1\in\im(\alpha)$, we have $1\in\im(\alpha u e_0(b_0a_0)^{k-2})$.
In addition, if $0\in\dom(\alpha)$, then $0\in\im(\alpha u e_0(b_0a_0)^{k-2})$, since $0\alpha=0$ and $0u e_0(b_0a_0)^{k-2}=0$,
whence $\alpha u e_0(b_0a_0)^{k-2}\in D_1$ and so
$$
wc_0 \sim w_{\alpha u e_0(b_0a_0)^{k-2}} c_0 \sim w_{\alpha u e_0(b_0a_0)^{k-2}} c_0 w_{\sigma_1}\in W_8.
$$
On the other hand, suppose that $0\notin\dom(\alpha)$. Then, $0\notin\dom(\alpha u e_0(b_0a_0)^{k-2})$.
If $\{1\}\subsetneq \im(\alpha u e_0(b_0a_0)^{k-2})$, then we have again $\alpha u e_0(b_0a_0)^{k-2}\in D_1$, whence
$$
wc_0 \sim w_{\alpha u e_0(b_0a_0)^{k-2}} c_0 \sim w_{\alpha u e_0(b_0a_0)^{k-2}} c_0 w_{\sigma_1}\in W_8.
$$
If $\im(\alpha u e_0(b_0a_0)^{k-2})=\{1\}$, then
$$
w c_0\sim w_{\alpha u e_0(b_0a_0)^{k-2}} c_0\sim w',
$$
for some $w' \in W_\mathrm{s}$, by a previous case (just like in (\ref{w=1})).

\smallskip

At last, we suppose that $w \in W_9$ and study each case for $x \in \{a_{0},b_{0},e_{0},f_{0},d,z,c_{0}\}$.
Let $0\leqslant \ell\leqslant n-2$ and let $\alpha\in D$ be such that $w=w_\alpha c_0 z b_0^\ell$.

\smallskip

\noindent\textsc{case} $x=a_{0}$.

By observing that, $za_0=zb_0$, $zb_0a_0=z$ and $zb_0^\ell a_0=zb_0^\ell$, for $2\leqslant\ell\leqslant n-2$, we obtain
$$
wa_0=w_\alpha c_0 z b_0^\ell a_0 \sim \left\{
\begin{array}{ll}
 w_\alpha c_0 z  b_0 \in W_9 & \mbox{if $\ell=0$}\\
 w_\alpha c_0 z  \in W_9  & \mbox{if $\ell=1$}\\
 w_\alpha c_0 z  b_0^\ell   \in W_9 & \mbox{if $2\leqslant\ell\leqslant n-2$.}
\end{array}\right.
$$

\smallskip

\noindent\textsc{case} $x=b_0$.

In this case, we have $wb_0=w_\alpha c_0 z b_0^{\ell+1}\in W_9$, for $0\leqslant\ell<n-2$, and
$wb_0 = w_\alpha c_0 z b_0^{n-1}\sim w_\alpha c_0 z\in W_9$, for $\ell=n-2$.

\smallskip

\noindent\textsc{case} $x=e_0$.

Since $ze_0=zb_0e_0=z$ and $zb_0^\ell e_0= zb_0^\ell$, for $2\leqslant\ell\leqslant n-2$, we get
$$
we_0=w_\alpha c_0 z b_0^\ell e_0 \sim \left\{
\begin{array}{ll}
 w_\alpha c_0 z  \in W_9  & \mbox{if $\ell=0,1$}\\
 w_\alpha c_0 z  b_0^\ell   \in W_9 & \mbox{if $2\leqslant\ell\leqslant n-2$.}
\end{array}\right.
$$

\smallskip

\noindent\textsc{case} $x=f_0$.

Now, observe that $zb_0^\ell f_0 = zb_0^\ell$, for $1\leqslant\ell\leqslant n-2$, and $zf_0=dz$. Hence,
$$
wf_0=w_\alpha c_0 z b_0^\ell f_0 \sim \left\{
\begin{array}{ll}
 w_\alpha c_0 dz \sim   w_\alpha f_0 dz \sim   w_{\alpha f_0 dz} \in W_s  & \mbox{if $\ell=0$}\\
 w_\alpha c_0 z  b_0^\ell   \in W_9 & \mbox{if $1\leqslant\ell\leqslant n-2$.}
\end{array}\right.
$$

\smallskip

\noindent\textsc{case} $x=d$.

Notice that, $zb_0^\ell d = (f_0b_0)^{n-1}f_0zb_0^\ell$ and, by Lemma \ref{PwEnd},
there exists a word $u$ on the alphabet $\{a_0,b_0,e_0,f_0\}$ such that (as transformation) $u\in \PT_{n-2}^{0,1}$ and
$ c_0 (f_0b_0)^{n-1}f_0 \sim u c_0$. Then, as clearly $\alpha u\in D$, we have
$$
wd = w_\alpha c_0 z b_0^\ell d \sim  w_\alpha c_0 (f_0b_0)^{n-1}f_0zb_0^\ell \sim w_\alpha u c_0 zb_0^\ell \sim w_{\alpha u} c_0 zb_0^\ell \in W_9.
$$

\smallskip

\noindent\textsc{case} $x=z$.

In order to deal with this case, let us recall that we already proved above that, if $w_1\in W_8$ and $y\in\{b_0,e_0\}$, then there exists $w_2\in W_8$ such that $w_1y\sim w_2$. It follows immediately that, if $w_1\in W_8$ and $v\in\{b_0,e_0\}^*$, then there exists $w_2\in W_8$ such that $w_1v\sim w_2$.

Now, observe that, $b_0^\ell z=z$, $z^2=(e_0b_0)^{n-3}e_0$ and,
as $D\subset D_1$, we have $w_\alpha c_0\in W_8$. Thus,
$$
wz = w_\alpha c_0 z b_0^\ell z \sim w_\alpha c_0 z^2 \sim  w_\alpha c_0 (e_0b_0)^{n-3}e_0 \sim w',
$$
for some $w'\in W_8$.

\smallskip

\noindent\textsc{case} $x=c_{0}$.

If $\ell=0$, then
$$
w c_0=  w_\alpha c_{0}zc_{0} \sim w_\alpha zc_{0} \sim w_\alpha z^{2}c_{0} \sim w_{\alpha z^{2}}c_{0}w_{\sigma_1} \in W_8,
$$
since $\alpha z^2 \in D_1$.

Finally, for $1\leqslant\ell\leqslant n-2$, we have $zb_0^\ell =zb_0^\ell f_0$ and so
$$
w c_0= w_\alpha c_{0}zb_{0}^\ell c_0 \sim w_\alpha c_{0}zb_{0}^\ell f_0 c_0 \sim w_\alpha c_{0}zb_{0}^\ell f_0  \sim w_\alpha c_{0}zb_{0}^\ell =w\in W_9.
$$

This concludes the proof.
\end{proof}

Since all relations of $R_\mathrm{w}$ are satisfied by the generators $a_0$, $b_0$, $e_0$, $f_0$, $d$, $c_0$ and $z$ of $\PwEnd(S_n)$,
bearing in mind Lemmas \ref{PwEnd2} and \ref{PwEnd1}, we can now conclude the following result.

\begin{theorem}\label{presPw}
The presentation $\langle a_0,b_0,e_0,f_0,c_0,d,z\mid R_\mathrm{w}\rangle$ defines the monoid $\PwEnd(S_n)$.
\end{theorem}

By taking the set $R_\mathrm{s}$ of defining relations for $\PsEnd(S_n)$ based on the presentations by Popova for $\PT(\Omega_{n-1})$
\cite{Popova:1961} and by Moore for $\Sym(\Omega_{n-1})$ \cite{Moore:1897}, we have $|R_\mathrm{w}|=n+40$.

\section{A presentation for $\IEnd(S_n)$}

Let $n\geqslant 4$ and consider the alphabet $A=\{a_0,b_0,e_1,d,z_1\}$ and the set $R$ formed by the following $3n+9$ monoid relations:
\begin{enumerate}
\item[$(R_1)$] $a_0^2=1$;
\item[$(R_2)$] $b_0^{n-1}=1$;
\item[$(R_3)$] $(b_0a_0)^{n-2}=1$;
\item[$(R_4)$] $(a_0b_0^{n-2}a_0b_0)^3=1$;
\item[$(R_5)$] $(a_0b_0^{n-1-j}a_0b_0^j)^2=1$, $j=2,\ldots,n-3$;
\item[$(R_6)$] $e_1^2=e_1$ and $d^2=d$;
\item[$(R_7)$] $a_0e_1=e_1a_0$, $da_0=a_0d$, $db_0=b_0d$ and $de_1=e_1d$;
\item[$(R_{8})$] $b_0a_0b_0^{n-2}e_1b_0a_0b_0^{n-2}=b_0^{n-2}e_1b_0$;
\item[$(R_{9})$] $(b_0^{n-2}e_1b_0a_0)^2=(a_0b_0^{n-2}e_1b_0)^2$;
\item[$(R_{10})$] $b_0^{n-2}e_1b_0a_0b_0^{n-2}e_1b_0=(a_0b_0^{n-2}e_1b_0)^2$;
\item[$(R_{11})$] $z_1^3=z_1$;
\item[$(R_{12})$] $z_1b_0=z_1a_0$;
\item[$(R_{13})$] $b_0^{n-2}z_1=a_0z_1$;
\item[$(R_{14})$] $a_0b_0^jz_1=b_0^jz_1$, $j=1,\ldots,n-3$;
\item[$(R_{15})$] $e_1b_0z_1=z_1d$;
\item[$(R_{16})$] $e_1b_0^jz_1=b_0^jz_1$, $j=2,\ldots,n-3$;
\item[$(R_{17})$] $(e_1b_0)^{n-3}e_1=z_1^2a_0b_0^{n-4}$;
\item[$(R_{18})$] $dz_1^2=z_1a_0z_1$;
\item[$(R_{19})$] $(dz_1)^2=dz_1d$.
\end{enumerate}

The monoid $\PAut(S_n)$ is defined by the presentation $\langle A\mid R\rangle$, as proved by Fernandes and Paulista in \cite{Fernandes&Paulista:2023}.
Observe that $R$ contains a set of defining relations of $\I(\Omega_{n-1})\zeta$.

\smallskip

Next, let $B=A\cup\{c\}=\{a_0,b_0,e_1,c,d,z_1\}$ and let $\bar R$ be the set of relations $R$ together with the following $10$ relations on the alphabet  $B$:
\begin{enumerate}
\item[$(R_{20})$]  $cz_1=z_1^2d$;
\item[$(R_{21})$]  $z_1c=z_1^2b_0^{n-2}e_1b_0$;
\item[$(R_{22})$]  $ca_0z_1=a_0ca_0z_1^2$;
\item[$(R_{23})$]  $ce_1=e_1c$;
\item[$(R_{24})$]  $dc=cb_0^{n-2}e_1b_0=c$;
\item[$(R_{25})$]  $cd=db_0^{n-2}e_1b_0=c^2$;
\item[$(R_{26})$]  $cb_0a_0=b_0a_0c$ and $cb_0^{n-2}a_0b_0=b_0^{n-2}a_0b_0c$.
\end{enumerate}

Notice that, it is easy to verify that all relations of $\bar R$ are satisfied by the generators $B$ of $\IEnd(S_n)$.

\smallskip

Let $W_\mathrm{A} \subseteq \{a_0,b_0,e_1,d,z_1\}^*$ be a set of canonical forms for $\PAut(S_n)$ on the generators $a_0, b_0, e_1, d$ and $z_1$.

Let $Q$ be the set of all sequences $(i_{0},i_{1},\ldots,i_{k}\mid j_{1},\ldots,j_{k})$ of elements in $\Omega_{n-1}$ such that
$k\in \{1,\ldots,n-2\}$, the elements $i_{0},i_{1},\ldots,i_{k}$ are pairwise distinct and $j_{1}<\cdots<j_{k}$.

Let $q=(i_{0},i_{1},\ldots,i_{k}\mid j_{1},\ldots,j_{k})\in Q$. Then, we put
\begin{enumerate}
\item $j_{q}=\min(\Omega_{n-1}\setminus\{1,j_{2},\ldots,j_{k}\})$;

\item $w_{q}'$ to be the word in $W_\mathrm{A}$ representing the transformation
$\transf{
0 & i_{0} & i_{1} & i_{2} & \cdots  & i_{k} \\
0 & 1 & j_{q} & j_{2} & \cdots  & j_{k}
} \in \PAut(S_n)$;

\item $w_{q}''$ to be the word in $W_\mathrm{A}$ representing the transformation
$\transf{
0 & j_{q} & j_{2} & \cdots & j_{k} \\
0 & j_{1} & j_{2} & \cdots & j_{k}
} \in \PAut(S_n)$.
\end{enumerate}

Consider the following sets of words on the alphabet $\{a_0,b_0,e_1,c,d,z_1\}$:
$$
W_{10} = \{dw_{q}'cw_{q}'' \mid q \in Q\}
\quad\text{and}\quad
W_\mathrm{I} = W_\mathrm{A} \cup W_{10}.
$$

Notice that, it is a routine matter to show that the mapping $Q\rightarrow W_{10}$, $q\mapsto dw_{q}'cw_{q}''$, for $q\in Q$, is a bijection.
On the other hand, clearly, a typical element $w$ of $K_0=\{\alpha\in\I(\Omega_{n-1}^0)\mid 0\not\in\dom(\alpha), 0\in\im(\alpha)$ and $|\im(\alpha)|\geq 2\}$
can be represented in the form
$$
w=\left(
\begin{array}{ccccc}
i_{0} & i_{1} & i_{2} & \cdots & i_{k} \\
0 & j_{1} & j_{2} & \cdots & j_{k}
\end{array}%
\right)
$$
with $q=(i_{0},i_{1},\ldots,i_{k}\mid j_{1},\ldots,j_{k})\in Q$
(and so, as transformations, $w=dw_{q}'cw_{q}''$),
whence $|K_0|=|Q|=|W_{10}|$.
Therefore, as $\IEnd(S_n)$ is the disjoint union of $\PAut(S_n)$ and
$K_0$, we immediately have the following lemma.

\begin{lemma}\label{IEnd3}
$|W_\mathrm{I}| = |\IEnd(S_{n})|$.
\end{lemma}

\smallskip

Now, let $\I_{n-2}^{0,1}=\{\alpha\in\I(\Omega_{n-1}^0)\mid 0\alpha=0, 1\alpha=1, \{2,\ldots,n-1\}\alpha\subseteq\{2,\ldots,n-1\}\}$
and recall the permutations $a'_0=b_0^{n-2}a_0b_0$ and $b'_0=b_0a_0$ of $\Omega_{n-1}^0$.
Clearly, $\I_{n-2}^{0,1}$ is a submonoid of $\I(\Omega_{n-1})\zeta$ isomorphic to $\I(\Omega_{n-2})$ that admits
$\{a'_0,b'_0,e_1\}$ as a generating set.
Then, since $\bar{R}$ contains a set of defining relations of $\I(\Omega_{n-1})\zeta$,
in view of the relations $R_{23}$ and $R_{26}$, similarly to Lemma \ref{r1}, we can prove the following lemma:
\begin{lemma}\label{IEnd2}
Let $w$ be a word on the alphabet $\{a_{0},b_{0},e_{1}\}$ such that (as transformation) $w\in\I_{n-2}^{0,1}$.
Then, $cw \sim_{\bar{R}} wc$.
\end{lemma}

\begin{lemma}\label{IEnd1}
$b_{0}^{n-2}e_{1}b_{0}c \sim_{\bar{R}} c^{2}$.
\end{lemma}
\begin{proof}
Since $b_0^{n-2}e_1b_0d=\transf{2&\cdots&n-1\\2&\cdots&n-1}=db_0^{n-2}e_1b_0$, we have
$$
b_{0}^{n-2}e_{1}b_{0}c \sim_{\bar{R}} b_{0}^{n-2}e_{1}b_{0}dc \sim_{\bar{R}} db_{0}^{n-2}e_{1}b_{0}c \sim_{\bar{R}} cdc \sim_{\bar{R}} c^{2},
$$
as required.
\end{proof}

At this stage, recall that the mapping $\zeta:\PT(\Omega_{n-1}^0)\longrightarrow\PT(\Omega_{n-1}^0)$, $\alpha\longmapsto\zeta_\alpha$,
is defined by $\dom(\zeta_\alpha)=\dom(\alpha)\cup\{0\}$, $0\zeta_\alpha=0$ and $\zeta_\alpha|_{\Omega_{n-1}}=\alpha|_{\Omega_{n-1}}$,
for any $\alpha\in\PT(\Omega_{n-1}^0)$.

For $\alpha \in \PAut(S_{n})$, let $w_\alpha$ be the word in $W_\mathrm{A}$ representing the transformation $\alpha$.
If $\alpha$ is the partial identity $\id_Y$ for some subset $Y \subseteq \Omega_{n-1}^{0}$, then we also write $w_Y$ for $w_\alpha$.

\begin{lemma}\label{IEnd4}
Let $w\in W_\mathrm{I}$ and let $x\in B$. Then, there exists a word $w'\in W_\mathrm{I}$ such that $wx \sim_{\bar R} w'$.
\end{lemma}
\begin{proof}
In this proof, for $0\leqslant i,j\leqslant n-1$, we denote by $(i\;j)$ the transposition of $\Omega_{n-1}^0$ that applies $i$ into $j$.
Also, the congruence $\sim_{\bar{R}}$ is denoted simply by $\sim$.

\smallskip

First, let us suppose that $w\in W_\mathrm{A}$ and $x\in A$.
Then, there exists $w'\in W_\mathrm{A}\subseteq W_\mathrm{I}$ such that $wx \sim w'$.
Thus, suppose that $x=c$.
We will consider six cases for $w$ (as transformation).

\smallskip

\noindent\textsc{case 1.} $0 \in \dom(w)$, $0w=0$, $\Omega_{n-1}w\subseteq \Omega_{n-1}$ and $1\in \im(w)$.
If $\im(w)=\{0,1\}$, then $w=b_0^iz_1^2$, for some $0\leqslant i\leqslant n-2$, whence
$$
wc\sim b_0^iz_1^2 c \sim b_0^i z_1^2 b_0^{n-2}e_1b_0 \sim w',
$$
for some $w'\in W_A$.
So, suppose that $|\im(w)|=r\geqslant 3$ and let
$
w=\transf{
0&i_{0} & i_{1} & i_{2} & \cdots & i_{r-2} \\
0 &1 & j_{1} & j_{2} & \cdots & j_{r-2}
}
$,
with $j_{1} < j_{2} < \cdots < j_{r-2}$.
Take $q=(i_{0},i_{1},\ldots,i_{k}\mid j_{1},\ldots,j_{k})$. Then, $q\in Q$ and $j_1,j_q\geqslant 2$.
Now, observe that $w (j_1\; j_q) = \transf{
0&i_{0} & i_{1} & i_{2} & \cdots & i_{r-2} \\
0 &1 & j_{q} & j_{2} & \cdots & j_{r-2}
}$
and
$(j_1\; j_q) \id_{\im(w)}= \transf{
0 & 1 & j_{q} & j_{2} & \cdots & j_{r-2} \\
0 &1  & j_{1} & j_{2} & \cdots & j_{r-2}
}$,
whence
$w w_{(j_1\; j_q)}\sim w'_q$ and $b_0^{n-2}e_1b_0 w_{(j_1\; j_q)} w_{\im(w)} \sim w''_q$.
In addition, by Lemma \ref{IEnd2}, we have $w_{(j_1\; j_q)} w_{\im(w)} c \sim c w_{(j_1\; j_q)} w_{\im(w)}$.
Therefore,
\begin{align*}
wc \sim w w_{(j_1\; j_q)} w_{(j_1\; j_q)}  w_{\im(w)} c \sim
w w_{(j_1\; j_q)} c w_{(j_1\; j_q)}  w_{\im(w)}  \sim
w w_{(j_1\; j_q)} d c b_0^{n-2}e_1b_0 w_{(j_1\; j_q)} \sim  \qquad \\  \sim
dw w_{(j_1\; j_q)}  c b_0^{n-2}e_1b_0 w_{(j_1\; j_q)} \sim d w'_q c w''_q \in W_{10}.
\end{align*}

\smallskip

\noindent\textsc{case 2.} $0 \in \dom(w)$, $0w=0$, $\Omega_{n-1}w\subseteq \Omega_{n-1}$ and $1\notin \im(w)$. Using Lemma \ref{IEnd1}, we obtain
$$wc \sim wdc \sim dwc \sim dwb_{0}^{n-2}e_{1}b_{0}c \sim dwc^{2} \sim dwdb_{0}^{n-2}e_{1}b_{0} \sim w',$$
for some $w' \in W_\mathrm{A}$.

\smallskip

\noindent\textsc{case 3.} $0 \notin \dom(w)$ and $\Omega_{n-1}w\subseteq \Omega_{n-1}$, i.e. $0 \notin \im(w)$.
In this case, we have $w = d\zeta_w$, where $\zeta_w \in \PAut(S_n)$, $0 \in \dom(\zeta_w)$, $0\zeta_w = 0$ and $\Omega_{n-1}\zeta_w \subseteq \Omega_{n-1}$. If $1\in \im(\zeta_w)$, then from \textsc{case 1}, we obtain
$w_{\zeta_w} c \sim dw_q'cw_q'' \in W_{10}$,
for some $q\in Q$. Thus,
$$
wc \sim d w_{\zeta_w} c \sim d^2w_q'cw_q'' \sim dw_q'cw_q'' \in W_{10}.
$$
So, suppose that $1\notin \im(\zeta_w)$. Then, from \textsc{case 2}, we have that there exists $u \in W_\mathrm{A}$ such that $w_{\zeta_w} c \sim u$. Thus,
$$
wc \sim d w_{\zeta_w} c \sim du \sim w',
$$
for some $w' \in W_\mathrm{A}$.

\smallskip

\noindent\textsc{case 4.} $0 \notin \dom(w)$ and $\im(w)=\{0\}$. Here, we have $w \sim wz_{1}^{2}$ and thus
$$
wc \sim wz_{1}^{2}c \sim wz_{1}z_{1}^{2}b_{0}^{n-2}e_{1}b_{0} \sim w',
$$
for some $w' \in W_\mathrm{A}$.

\smallskip

\noindent\textsc{case 5.} $0 \in \dom(w)$, $0w=1$ and $\Omega_{n-1}w \subseteq \{0\}$. We have again $w \sim wz_{1}^{2}$ and so
$$wc \sim wz_{1}^{2}c \sim wz_{1}z_{1}^{2}b_{0}^{n-2}e_{1}b_{0} \sim w',$$
for some $w' \in W_\mathrm{A}$.

\smallskip

\noindent\textsc{case 6.} $0\in \dom(w)$, $0w > 1$ and $\Omega_{n-1}w \subseteq \{0\}$. From $w \sim wb_{0}^{n-2}e_{1}b_{0}$ and Lemma \ref{IEnd1}, we have
$$wc \sim wb_{0}^{n-2}e_{1}b_{0}c \sim wc^{2} \sim wdb_{0}^{n-2}e_{1}b_{0} \sim w',$$
for some $w' \in W_\mathrm{A}$.

\smallskip

Next, we suppose that $w\in W_{10}$ and consider each case for $x\in B$.
Take $q=(i_{0},i_{1},\ldots,i_{k}\mid j_{1},\ldots,j_{k})\in Q$ such that $w = dw_q'cw_q''$.
Observe that, as transformation, we have
$
w=\transf{
i_{0} & i_{1} & i_{2} & \cdots & i_{k} \\
0 & j_{1} & j_{2} & \cdots & j_{k}
}
$.

\smallskip

\noindent\textsc{case} $x=a_{0}$.

Suppose that $j_{1}=1$ and $j_{2}=2$. Then, $j_{q}>2$ and, clearly, $w_{q}''a_{0}\sim w_{(j_{2}\;j_{q})}w_{q}''$.
On the other hand, by Lemma \ref{IEnd2}, we have $c w_{(j_{2}\;j_{q})} \sim w_{(j_{2}\;j_{q})} c$.
Take $q'=(i_{0},i_{2},i_1,i_3,\ldots,i_{k}\mid j_{1},\ldots,j_{k})$. Then, $q'\in Q$, $j_{q'}=j_q$, $w''_{q'}=w''_q$ and $w'_q w_{(j_{2}\;j_{q})}\sim w'_{q'}$.
Thus, we get
$$
wa_0 = dw_{q}'cw_{q}''a_{0} \sim dw_{q}'cw_{(j_{2}\;j_{q})}w_{q}'' \sim dw_{q}'w_{(j_{2}\;j_{q})}cw_{q}'' \sim dw'_{q'}cw''_{q'}\in W_{10}.
$$

Now, suppose that $j_{1}=1$ and $j_{2}>2$. Take $q'=(i_{0},i_1,i_2,\ldots,i_{k}\mid 2,j_{2},\ldots,j_{k})$.
Then, $q'\in Q$, $j_{q'}=j_q=2$, $w'_{q'}=w'_q$ and $w''_qa_0 \sim w''_{q'}$. Thus, we obtain
$$
wa_{0}=dw_{q}'cw_{q}''a_{0} \sim dw'_{q'}cw''_{q'}\in W_{10}.
$$

Next, suppose that $j_{1}=2$. Then, for $q'=(i_{0},i_1,i_2,\ldots,i_{k}\mid 1,j_{2},\ldots,j_{k})$,
we have $q'\in Q$, $j_{q'}=j_q=2$, $w'_{q'}=w'_q$ and $w''_qa_0 \sim w''_{q'}$, whence
$$
wa_{0}=dw_{q}'cw_{q}''a_{0} \sim dw'_{q'}cw''_{q'}\in W_{10}.
$$

Finally, suppose that $j_{1}>2$. Then, $w_{q}'' = w_{q}''a_{0}$ (as transformations) and thus
$$
wa_0 = dw_{q}'cw_{q}''a_{0} \sim dw_{q}'cw_{q}'' \in W_{10}.
$$

\smallskip

\noindent\textsc{case} $x=b_0$.

Let $\sigma = b_0a_0 = (2\;3\;\cdots\;n-1)\in\Sym(\Omega_{n-1}^0)$ and let $2\leqslant i<j\leqslant n-1$.
Clearly, $\sigma^{n-2}w_{q}''=w_{q}''$ (as transformations) and $(i\;j)^{2}\sigma =\sigma$. Moreover, by Lemma \ref{IEnd2},
we have $cw_{\sigma}\sim w_{\sigma}c$ as well as $cw_{(i\;j)}\sim w_{(i\;j)}c$.

Suppose that $j_{1}=1$ and $j_{k}=n-1$.
Then, $j_{q} < n-1$. Take $q'=(i_{0},i_{k},i_{1},\ldots,i_{k-1}\mid 1,2,j_{2}+1,\ldots,j_{k-1}+1)$.
Hence, $q'\in Q$, $j_{q'}=j_q+1$, $w_{q}'w_{\sigma}w_{(2\;j_{q}+1)}\sim w'_{q'}$
and $w_{(2\;j_{q}+1)}w_{\sigma}^{n-3}w_{q}''b_{0} \sim w''_{q'}$.
Thus, we obtain
\begin{align*}
wb_{0} = dw_{q}'cw_{q}''b_{0}\sim dw_{q}'cw_{\sigma}^{n-2}w_{q}''b_{0}\sim dw_{q}'w_{\sigma}cw_{(2\;j_{q}+1)}^{2}w_{\sigma}^{n-3}w_{q}''b_{0}
\sim  \qquad \\  \sim dw_{q}'w_{\sigma}w_{(2\;j_{q}+1)}cw_{(2\;j_{q}+1)}w_{\sigma}^{n-3}w_{q}''b_{0} \sim dw_{q'}'cw_{q'}''\in W_{10}.
\end{align*}

Suppose that $j_{1}>1$ and $j_{k}=n-1$. Then, $j_{q}=2$. Take $q'=(i_{0},i_{k},i_{1},\ldots,i_{k-1}\mid 1,2,j_{2}+1,\ldots,j_{k-1}+1)$.
Then, $q'\in Q$, $j_{q'}=3$, $w_{q}'w_{\sigma}w_{(j_{1}+1\;3)} \sim w'_{q'}$ and $w_{(j_{1}+1\;3)}w_{\sigma}^{n-3}w_{q}''b_{0} \sim w''_{q'}$.
Thus, we get
\begin{align*}
wb_{0} = dw_{q}'cw_{q}''b_{0} \sim dw_{q}'cw_{\sigma}^{n-2}w_{q}''b_{0} \sim dw_{q}'w_{\sigma}cw_{(j_{1}+1\;3)}^{2}w_{\sigma}^{n-3}w_{q}''b_{0}
\sim  \qquad \\  \sim  dw_{q}'w_{\sigma}w_{(j_{1}+1\;3)}cw_{(j_{1}+1\;3)}w_{\sigma}^{n-3}w_{q}''b_{0} \sim dw_{q'}'cw_{q'}''\in W_{10}.
\end{align*}

Finally, suppose that $j_{k}<n-1$. Then, being $q'=(i_{0},i_{1},\ldots,i_{k}\mid j_{1}+1,\ldots,j_{k}+1)$, we have $q'\in Q$, $j_{q'}=j_q+1$,
$w_{q}'w_{\sigma}w_{(2\;j_{q}+1)}\sim w'_{q'}$ and $w_{(2\;j_{q}+1)}w_{\sigma}^{n-3}w_{q}''b_{0} \sim w''_{q'}$, whence
\begin{align*}
wb_{0} = dw_{q}'cw_{q}''b_{0} \sim dw_{q}'cw_{\sigma}^{n-2}w_{q}''b_{0} \sim dw_{q}'w_{\sigma}cw_{(2\;j_{q}+1)}^{2}w_{\sigma}^{n-3}w_{q}''b_{0}
\sim  \qquad \\  \sim  dw_{q}'w_{\sigma}w_{(2\;j_{q}+1)}cw_{(2\;j_{q}+1)}w_{\sigma}^{n-3}w_{q}''b_{0} \sim dw_{q'}'cw_{q'}'' \in W_{10}.
\end{align*}

\smallskip

\noindent\textsc{case} $x=e_1$.

Suppose that $k=1$ and $j_{1}=n-1$. Then, $j_{q}=2$ and $w_{q}''e_{1}=z_{1}^{2}w_{q}''$ (as transformations).
Hence,
$$
we_{1}=dw_{q}'cw_{q}''e_{1} \sim dw_{q}'cz_{1}^{2}w_{q}'' \sim dw_{q}'z_{1}^{2}dz_{1}w_{q}'' \sim w',
$$
for some $w' \in W_\mathrm{A}$.

Next, suppose that $k>1$ and $j_{k}=n-1$.
Then, $w_{q}''e_{1}=e_{1}w_{q}''=e_1^2w_q''$ (as transformations).
Let us take $q'=(i_{0},i_{1},\ldots ,i_{k-1}\mid j_{1},\ldots ,j_{k-1})$.
Then, $q'\in Q$, $j_{q'}=j_q$, $w_{q}'e_{1} \sim w_{q'}'$ and $e_{1}w_{q}'' \sim w_{q'}''$. Thus, we obtain
$$
we_{1}=dw_{q}'cw_{q}''e_{1} \sim dw_{q}'ce_{1}^{2}w_{q}'' \sim dw_{q}'e_{1}ce_{1}w_{q}'' \sim dw_{q'}'cw_{q'}'' \in W_{10}.
$$

Finally, suppose that $j_{k}<n-1$. Then, $w_{q}''=w_{q}''e_{1}$ (as transformations). So, we have
$$
we_{1}=dw_{q}'cw_{q}''e_{1} \sim dw_{q}'cw_{q}'' \in W_{10}.
$$

\smallskip

\noindent\textsc{case} $x=c$.

Obviously, $w_{q}''d=dw_{q}''$ as well as $w_{q}''(2\;j_{2})^{2}=w_{q}''$ (as transformations).
Moreover,  by Lemma \ref{IEnd2}, we have $w_{(2\;j_{2})}c\sim cw_{(2\;j_{2})}$.

Suppose that $j_{1}=1$ and take $q'=(i_{1},\ldots,i_{k}\mid j_{2},\ldots,j_{k})$.
Then, $q'\in Q$, $j_{q'}=2$, $w_{q}'db_{0}^{n-2}e_{1}b_{0}w_{q}''w_{(2\;j_{2})} \sim w_{q'}'$ and $w_{(2\;j_{2})}b_{0}^{n-2}e_{1}b_{0} \sim w_{q'}''$.
Thus, we get
\begin{align*}
wc = dw_{q}'cw_{q}''c \sim dw_{q}'cw_{q}''dc \sim dw_{q}'cdw_{q}''c \sim dw_{q}'db_{0}^{n-2}e_{1}b_{0}w_{q}''c \sim dw_{q}'db_{0}^{n-2}e_{1}b_{0}w_{q}''cb_{0}^{n-2}e_{1}b_{0} \sim  \qquad \\  \sim
dw_{q}'db_{0}^{n-2}e_{1}b_{0}w_{q}''w_{(2\;j_{2})}^{2}cb_{0}^{n-2}e_{1}b_{0} \sim dw_{q}'db_{0}^{n-2}e_{1}b_{0}w_{q}''w_{(2\;j_{2})}cw_{(2\;j_{2})}b_{0}^{n-2}e_{1}b_{0} \sim dw_{q'}'cw_{q'}'' \in W_{10}.
\end{align*}

Now, suppose that $j_{1}>1$. Then, $w_{q}''c=dw_{q}''$ (as transformations) and so
$$
wc=dw_{q}'cw_{q}''c \sim dw_{q}'cdw_{q}'' \sim dw_{q}'db_{0}^{n-2}e_{1}b_{0}w_{q}'' \sim w',
$$
for some $w' \in W_\mathrm{A}$.

\smallskip

\noindent\textsc{case} $x=d$.

Since $w_{q}''d=dw_{q}''$ (as transformations), we have
$$
wd = dw_{q}'cw_{q}''d \sim dw_{q}'cdw_{q}'' \sim dw_{q}'db_0^{n-2}e_1b_0w_{q}'' \sim w',
$$
for some $w' \in W_\mathrm{A}$.

\smallskip

\noindent\textsc{case} $x=z_1$.

Suppose that $1\notin \im(w_{q}'')$. Then, $w_{q}''z_{1}=z_{1}d$ (as transformations) and so
$$
wz_1 = dw_{q}'cw_{q}''z_{1} \sim dw_{q}'cz_{1}d \sim dw_{q}'z_{1}^{2}d^2 \sim w',
$$
for some $w' \in W_\mathrm{A}$.

Now, notice that $\id_{{\{0,1,2\}}}a_{0}z_{1}^{2}=a_{0}z_{1}^{2}$ and, by Lemma \ref{IEnd2},
we have $w_{(2\;j_{q})}c\sim cw_{(2\;j_{q})}$ and $w_{{\{0,1,2\}}}c \sim cw_{{\{0,1,2\}}}$.

Suppose that $1\in \im(w_{q}'')$ and $j_{q}=2$. Then, $w_{q}''z_{1}=a_{0}z_{1}$ (as transformations).
Take $q'=(i_{1},i_{0}\mid 1)$. Then, $q'\in Q$, $j_{q'}=2$, $w_{q}'a_{0}w_{\{0,1,2\}} \sim w_{q'}'$ and $a_{0}z_{1}^{2} \sim w_{q'}''$.
Thus, we obtain
$$
wz_1 = dw_{q}'cw_{q}''z_{1} \sim dw_{q}'ca_{0}z_{1} \sim dw_{q}'a_{0}ca_{0}z_{1}^{2} \sim dw_{q}'a_{0}cw_{\{0,1,2\}}a_{0}z_{1}^{2} \sim dw_{q}'a_{0}w_{\{0,1,2\}}ca_{0}z_{1}^{2} \sim dw_{q'}'cw_{q'}'' \in W_{10}.
$$

Finally, suppose that $1\in \im(w_{q}'')$ and $j_{q}>2$. Then, $w_{q}''z_{1}=w_{(2\;j_{q})}a_{0}z_1$ (as transformations).
Let us take $q'=(i_{1},i_{0}\mid 1)$. Hence, $q'\in Q$, $j_{q'}=2$, $w_{q}'w_{(2\;j_{q})}a_{0}w_{\{0,1,2\}} \sim w_{q'}'$ and $a_{0}z_{1}^{2} \sim w_{q'}''$.
Thus, we have
\begin{align*}
wz_1 = dw_{q}'cw_{q}''z_{1} \sim dw_{q}'cw_{(2\;j_{q})}a_{0}z_{1} \sim dw_{q}'w_{(2\;j_{q})}ca_{0}z_{1} \sim dw_{q}'w_{(2\;j_{q})}a_{0}ca_{0}z_{1}^{2}
\sim  \qquad \\  \sim
dw_{q}'w_{(2\;j_{q})}a_{0}cw_{\{0,1,2\}}a_{0}z_{1}^{2} \sim dw_{q}'w_{(2\;j_{q})}a_{0}w_{\{0,1,2\}}ca_{0}z_{1}^{2} \sim dw_{q'}'cw_{q'}'' \in W_{10}.
\end{align*}

This last case completes the proof.
\end{proof}

\color{black}

Since all $3n+19$ relations of $\bar{R}$ are satisfied by the generators $a_0,b_0,e_1,c,d,z_1$ of $\IEnd(S_n)$,
Lemmas \ref{IEnd3} and \ref{IEnd4} allow us to conclude the following result.

\begin{theorem}\label{presIE}
The presentation $\langle B\mid \bar R\rangle$ defines the monoid $\IEnd(S_n)$.
\end{theorem}

\color{black}

\bigskip

\lastpage

\end{document}